\documentclass[12pt]{article} 
 \usepackage[margin=1in]{geometry}
\usepackage{setspace}
\usepackage{graphicx}	
\usepackage{subcaption}
\usepackage{float} 
\usepackage[toc]{appendix}
\usepackage[draft]{todonotes} 
\usepackage{enumerate}	
\usepackage{paralist}	
\usepackage{amssymb}
\usepackage{amsmath, amsthm, amssymb}
\usepackage{thm-restate}

\usepackage[pdftex,hypertexnames=false,linktocpage=true,hidelinks]{hyperref}

\newtheorem{thm}{Theorem}[section]

\newtheorem{claim}{Claim}
\newtheorem{lemma}[thm]{Lemma}

\theoremstyle{definition}

\newtheorem{defn}[thm]{Definition}
\newtheorem{const}{Construction}

\def\F{\mathcal{F}}

\usepackage{tikz}
\usepackage{xcolor}

\usetikzlibrary{decorations.pathreplacing}

\usepackage[
backend=biber,
style= numeric
]{biblatex}

\title{Positive co-degree thresholds for spanning structures}

\author{Anastasia Halfpap
\footnote{Department of Mathematics, Iowa State University, Ames, IA, USA.  E-mail \texttt{ahalfpap@iastate.edu}.}
\and 
Van Magnan
\footnote{University of Vermont, Burlington, VT, USA. E-mail \texttt{vmagnan@uvm.edu}.}
}

\begin{document}

\maketitle

\begin{abstract}
The \textit{minimum positive co-degree} of a non-empty $r$-graph $H$, denoted $\delta_{r-1}^+(H)$, is the largest integer $k$ such that if a set $S \subset V(H)$ of size $r-1$ is contained in at least one $r$-edge of $H$, then $S$ is contained in at least $k$ $r$-edges of $H$. Motivated by several recent papers which study minimum positive co-degree as a reasonable notion of minimum degree in $r$-graphs, we consider bounds of $\delta_{r-1}^+(H)$ which will guarantee the existence of various spanning subgraphs in $H$. We precisely determine the minimum positive co-degree threshold for Berge Hamiltonian cycles in $r$-graphs, and asymptotically determine the minimum positive co-degree threshold for loose Hamiltonian cycles in $3$-graphs. For all $r$, we also determine up to an additive constant the minimum positive co-degree threshold for perfect matchings.

\end{abstract}

\section{Introduction}

We consider $r$-uniform hypergraphs i.e., structures $H = (V,E)$, where $V, E$ are sets such that $E \subseteq \binom{V}{r}$, where $\binom{V}{r}$ denotes the set of all $r$-element subsets of $V$. We shall often call a $r$-uniform hypergraph $H$ an \textit{r-graph} for short, and shall call $e \in E(H)$ an \textit{edge} of $H$, or an \textit{r-edge} of $H$ if the uniformity of $H$ may be ambiguous. We often denote an $r$-edge by the concatenation of its vertices, e.g., for a $3$-graph $H$ with $\{v_1,v_2,v_3\} \in V(H)$, we write $v_1v_2v_3$ (or $v_2v_1v_3$, etc.) for $\{v_1, v_2, v_3\}\in V(H)$.

For a $2$-graph $G$, the minimum degree $\delta(G)$ is a well-defined, important, and extensively studied parameter. Minimum degree problems for $r$-graphs are thus naturally motivated. However, when $r \geq 3$ there are multiple interpretations of ``degree''. A natural definition is based on \textit{vertex} degree: in an $r$-graph $H$, we define \[d_1(v) = |\{e : v \in e \in E(H)\}|\text{, and } \delta_1(H) = \underset{v \in V}{\min} \hspace{0.1 cm} d_1(v).\] 
However, other natural generalizations exist. In an $r$-graph, the {\it co-degree} of an $(r-1)$-set $S$ is the number of edges containing $S$. The minimum co-degree over all $(r-1)$-sets in $H$ is denoted $\delta_{r-1}(H)$. Note that when $r = 2$, both vertex degree and co-degree indeed return the standard notion of degree for $2$-graphs. However, in general, the two parameters behave quite differently. To these two definitions of degree, we add a third, more recent, parameter:

\begin{defn}
Given an $r$-graph $H$, the \textit{minimum positive co-degree} of $H=(V,E)$, denoted $\delta_{r-1}^+(H)$, is the largest integer $k$ such that, for any $S \in \binom{V}{r-1}$, either $d_{r-1}(S) = 0$ or $d_{r-1}(S) \geq k$. If $H$ is empty, we define $\delta_{r-1}^+(H) = 0$.

\end{defn}

We remark that this parameter is also referred to as the \textit{shadow degree} of a hypergraph in the literature by, for instance, Frankl and Wang in \cite{frankl2024intersectingfamilieslargeshadow}. Note that this introduces a slight abuse of notation for $d_1(v)$, as $v$ is not a set. More generally, if $S$ is a fixed set, we may slightly abuse notation by dropping set notation, e.g., we write $d_2(x,y)$ rather than $d_2(\{x,y\})$. Minimum positive co-degree is not quite a generalization of minimum degree: many graphs $G$ have $\delta(G) = 0$, but we always have $\delta_1^+(G) \geq 1$ unless $G$ is empty. However, $\delta(G) = \delta_1^+(G)$ as long as $G$ has no isolated vertices, which is not a particularly restrictive assumption to add. Moreover, in some ways, minimum positive co-degree seems to better mimic minimum degree than the ordinary minimum co-degree parameter. One advantage of minimum positive co-degree comes from the differing notions of independence in $r$-graphs. 

\begin{defn}
Given an $r$-graph $H$, a set $S \subset V(H)$ is \textit{independent} if no $r$-edge of $H$ is contained within $S$. We call $S$ \textit{strongly independent} if no $r$-edge of $H$ intersects $S$ in more than one element. 
\end{defn}
Note that for $r = 2$, there is no distinction between the notions of independence and strong independence, and the presence of a large independent set does not preclude a large minimum degree. However, for $r > 3$, the two types of independence interact very differently with ordinary minimum co-degree. An $r$-graph may contain large independent sets and maintain a large minimum co-degree; however, if $H$ is an $r$-graph with $r \geq 3$, and $H$ contains a strongly independent set $S$ with $|S| > 1$, then $\delta_{r-1}(H) = 0$. On the other hand, $r$-graphs with large strongly independent sets may have large minimum \textit{positive} co-degree. This behavior is particularly relevant because many natural hypergraph constructions (which often arise as analogs of well-known $2$-graph constructions) have large strongly independent sets; the adoption of minimum positive co-degree, as opposed to ordinary minimum co-degree, allows us to consider these natural constructions as potential extremal examples.

Balogh, Palmer, and Lemons~\cite{BLPcodegree} first considered minimum positive co-degree as a reasonable hypergraph analog of minimum degree, motivated both by previous minimum co-degree problems and by degree versions of the Erd\H{o}s-Ko-Rado Theorem.  Recently, Halfpap, Lemons, and Palmer~\cite{pcddensity} introduced the \textit{positive co-degree Tur\'an number}, denoted $\mathrm{co^+ex}(n, F)$, defined as the largest possible minimum positive co-degree in an $n$-vertex $r$-graph which does not contain $F$ a subhypergraph. 
The aim of this paper is to continue the exploration of minimum positive co-degree by studying minimum positive co-degree conditions which guarantee the existence of certain spanning structures, i.e., subgraphs on the same vertex set as the host graph. In particular, we examine the extent to which these results mimic minimum degree bounds for $2$-graphs.

A classical question in graph theory is to characterize minimum degree conditions under which graphs contain certain spanning structures as subgraphs. For instance, in an $n$-vertex graph $G$, a \textit{Hamiltonian cycle} is a cycle on $n$ vertices. A well-known theorem of Dirac gives a simple sufficient (though not necessary) degree condition for the existence of Hamiltonian cycles.

\begin{thm}[Dirac~\cite{dirac}]
Let $G$ be an $n$-vertex graph, with $n \geq 3$, and suppose $\delta(G) \geq \frac{n}{2}$. Then $G$ contains a Hamiltonian cycle. 

\end{thm}

Observe that the complete bipartite graph $K_{n/2 + 1, n/2 -1}$ has minimum degree $\frac{n}{2} - 1$ and does not contain a Hamiltonian cycle. Thus, we may say that $\frac{n}{2}$ is the \textit{degree threshold} for Hamiltonian cycles in graphs. Similarly, Dirac's Theorem implies that for $n$ even, any $n$-vertex graph $G$ with $\delta(G) \geq \frac{n}{2}$ contains a perfect matching. Again, a slightly unbalanced complete bipartite graph does not contain a perfect matching, so $\frac{n}{2}$ is also the degree threshold for perfect matchings in graphs.

Now, given an $r$-graph, we may analogously consider the degree threshold for Hamiltonian cycles. This will depend not only on our generalization of minimum degree, but also on our notion of hypergraph cycles. We make two standard definitions; always, for a $k$-vertex cycle, we understand subscripts to be taken modulo $k$.

\begin{defn}
In an $r$-graph $H$, a \textit{Berge cycle} of length $k$ is an ordered list of $k$ pairwise distinct vertices and $r$-edges from $H$
\[ (v_1,e_1,v_2,e_2, \dots, e_{k-1}, v_k, e_k, v_1)\]
such that for each $i \in \{1,2, \dots, k\}$, $v_i$ and $v_{i+1}$ are contained in $e_i$. 
\end{defn}

\begin{defn}

An $r$-uniform \textit{loose cycle} of length $k$ is an $r$-graph on $k(r-1)$ vertices (for some integer $k \geq 2$), say with vertex set $\{v_1, v_2, \dots , v_{k(r-1)}\}$, whose edges are those of the form
\[v_{i(r-1) + 1}v_{i(r-1) + 2}\dots v_{i(r-1) + r} \text{ for } 0 \leq i \leq k-1. \]

\end{defn}
\noindent Other notions of hypergraph cycles (e.g., tight cycles, $\ell$-cycles, $t$-tight Berge cycles; see \cite{Rödl2010_Book} for general discussion) exist, but we shall not consider them here. Note that depending upon our chosen definition of cycle, it may only be possible to find a Hamiltonian cycle in $H$ given certain divisibility criteria on $|V(H)|$. 


Various authors have considered vertex degree and co-degree thresholds for Hamiltonian cycles in $3$-graphs. Here, we give results on the exact co-degree and degree thresholds:

\begin{thm}[Czygrinow-Molla \cite{czygrinow2013tight}]
    Let $H$ be a $3$-graph on $n$ vertices, where $n$ is even and sufficiently large. If $\delta_2(H)\geq n/4$, then $H$ has a loose Hamiltonian cycle
\end{thm}

\begin{thm}[Han \cite{handissertation}]
    Let $H$ be a $3$-graph on $n$ vertices, where $n$ is even and sufficiently large. If $\delta_1(H)\geq \frac{7}{16}\binom{n}{2} $, then $H$ has a loose Hamiltonian cycle
\end{thm}

More general asymptotic results for $r$-graphs are known. With regards to loose Hamiltonian cycles in an $r$- graph $H$, K\"uhn, Mycroft, and Osthus in \cite{Kuhn2010general} give (among other things) an asymptotic threshold of $\displaystyle \delta_{r-1}(H)\approx\frac{n}{\lfloor \frac{r}{r-1} \rfloor (r-1)}$ for the existence of a loose Hamiltonian cycle.

To this broad body of results, we contribute exact and asymptotic positive co-degree thresholds. We begin by considering lower bounds. As previously mentioned, the degree threshold for both Hamiltonian cycles and perfect matchings in $2$-graphs is $\frac{n}{2}$. We can see that this threshold is tight by considering a slightly unbalanced complete bipartite graph; however, other constructions also achieve minimum degree $\frac{n}{2} - 1$ and avoid the specified spanning structures. For instance, consider the $2$-graph $H_{U,V}^2$, where $V(H) = U \sqcup V$, $N(u) = V(H) \setminus \{u\}$ for every $u \in U$, and $N(v) = U$ for every $v \in V$. If $|U| = \frac{n}{2} - 1$ and $|V| = \frac{n}{2} + 1$, then $H_{U,V}^2$ has no Hamiltonian cycle and no perfect matching. The following construction is an $r$-graph analogue of $H_{U,V}^2$, and will provide our primary source of lower bounds.

\medskip

\begin{const}\label{main construction}
$H_{U,V}^r$ is an $r$-graph with vertices in two classes $U$ and $V$, such that $|U| + |V| = n$. Every $(r-1)$ set $S \subset U$ is \textit{universal}, i.e., $N(S) = V(H_{U,V})\setminus S$. $V$ is a strongly independent set. Note that $\delta_{r-1}^+(H^r_{U,V}) = |U| - (r-2)$
\end{const}

\medskip 

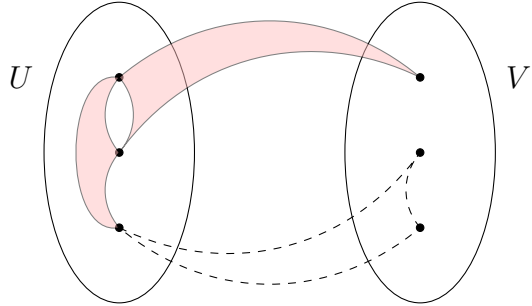
\begin{figure}[h]
\begin{center}
\begin{tikzpicture}



\draw (-2,0) ellipse [x radius=1, y radius=2];
\draw (-3,1) node[left]{$U$};

\draw (2,0) ellipse [x radius=1, y radius=2];
\draw (3,1) node[right]{$V$};

\filldraw (-2,0) circle (0.05 cm);
\filldraw (-2,1) circle (0.05 cm);
\filldraw (-2,-1) circle (0.05 cm);

\filldraw (2,0) circle (0.05 cm);
\filldraw (2,1) circle (0.05 cm);
\filldraw (2,-1) circle (0.05 cm);

\draw[fill=pink,opacity=0.5]
(-2,1) to[bend left=320] (-2,0) to[bend left=320] (-2,-1) to[bend left = 100] cycle
;

\draw[fill=pink,opacity=0.5]
(-2,1) to[bend right=320] (-2,0) to[bend right=320] (2,1) to[bend left = 320] cycle
;

;

\draw[dashed]
(2,-1) to[bend right=320] (2,0) to[bend right=320] (-2,-1) to[bend left = 320] cycle
;

\end{tikzpicture}

\caption{Two edges and a non-edge in $H^3_{U,V}$}
\end{center}

\end{figure}

Depending upon the sizes of $U$ and $V$ (and the value of $r$), $H_{U,V}^r$ may be easily seen to avoid certain spanning structures. 
If $|V| > \frac{n}{2}$, then $H_{U,V}^3$ does not contain either a Berge Hamiltonian cycle or a loose Hamiltonian cycle, since in any cyclic ordering of the vertices of $H_{U,V}^3$, two vertices of $V$ must be adjacent (and thus, using either model of cycle, contained in a $3$-edge together), while $V$ is strongly independent. More generally, $H_{U,V}^r$ has no Berge Hamiltonian cycle if $|V| > \frac{n}{2}$, and has no loose Hamiltonian cycle if $|V| > \frac{n}{r-1}$.


When $|V| = \lceil \frac{n+1}{2} \rceil$, $H_{U,V}^2$ is an example of a construction which optimizes minimum degree while avoiding a Hamiltonian cycle; it is natural to conjecture that $H_{U,V}^r$ optimally avoids Hamiltonian cycles for the appropriate choice of $|V|$. In the cases which we are able to resolve, we shall see that this is essentially true. Firstly, we exactly determine the positive co-degree threshold for Berge cycles in all uniformities.

\begin{restatable}{thm}{bergetheorem} \label{Berge HC}
Let $r\ge 2$, $n \ge 6r-10$ and suppose $H$ is an $n$-vertex $r$-graph with $\delta_{r-1}^+(H) \geq \frac{n}{2} - r + 2$ and no isolated vertices. Then $H$ contains a Berge Hamiltonian cycle
\end{restatable}
Thus, if $|V| = \lceil \frac{n+1}{2} \rceil$, then $H_{U,V}^r$ optimizes minimum positive co-degree while avoiding a Berge Hamiltonian cycle.

We also asymptotically determine the minimum positive co-degree threshold for loose Hamiltonian cycles in $3$-graphs. Our bound shows that Construction~\ref{main construction} is asymptotically optimal.

\begin{restatable}{thm}{loosethm}\label{looseHC}
For any $0 < \varepsilon < \frac{1}{2}$, there exists $n_0 \in \mathbb{N}$ such that the following holds. Let $H$ be an $n$-vertex $3$-graph with $n \geq n_0$, $n$ even, and $\delta_2^+(H) \geq \left( \frac{1}{2} + \varepsilon \right)n$. Then $H$ contains a loose Hamiltonian cycle.
\end{restatable}

Unlike in the case of $2$-graphs, minimum degree thresholds for Hamiltonian cycles do not imply thresholds for $r$-uniform perfect matchings, in which each vertex is contained in a single $r$-edge, since neither a Berge Hamiltonian cycle nor a loose Hamiltonian cycle necessarily contains a perfect matching as a subhypergraph. (In fact, due to divisibility concerns, any $r$-graph containing a loose Hamiltonian cycle \textit{cannot} contain a perfect matching). Perhaps counterintuitively, the ordinary co-degree thresholds for $r$-uniform perfect matchings are often \textit{higher} than for various forms of Hamiltonian cycle: for any $r$, a co-degree threshold of (up to an additive constant depending only on $r$) $\frac{n}{2}$ for the existence of a perfect matching. By contrast, the co-degree threshold for $r$-uniform \textit{almost} perfect matchings becomes smaller as $r$ grows, and is (about) $\frac{n}{r}$. We direct readers towards the survey of R\"odl and Ruci\'nski\cite{Rödl2010_Book} for more precise statements and discussion of both.

We exactly determine the positive co-degree threshold for $3$-uniform perfect matchings, and find the positive co-degree threshold for $r$-uniform perfect matchings up to an additive constant. 

\begin{restatable}{thm}{threePM}\label{threePM}
Suppose $H$ is an $n$-vertex $3$-graph such that $3$ divides $n$,
$\delta_2^+(H) \geq \frac{2n}{3} - 1$, and $d_1(v) > 0$ for every $v \in V(H)$. Then $H$ contains a perfect matching.
\end{restatable}

\begin{restatable}{thm}{rPM}\label{rPM}
Let $r\geq 2, n\geq r^3+r^2-r$ where $r$ divides $n$, and suppose $H$ is an $n$-vertex $r$-graph with $\delta_{r-1}^+(H) \geq (\frac{r-1}{r}) n+r^2$ and $d_1(v) > 0$ for every $v \in V(H)$. Then $H$ contains a perfect matching.
\end{restatable}

We remark that the positive co-degree threshold for almost perfect matchings is close to the threshold given in Theorem~\ref{rPM}. In addition to proving Theorem~\ref{threePM} and Theorem~\ref{rPM}, we observe via Construction~\ref{main construction} that they are tight. Using Construction~\ref{main construction}, we can also exhibit $n$-vertex $r$-graphs with minimum positive co-degree  $\left( \frac{r-1}{r}  - \varepsilon \right)n$ which do not admit a matching of size $\left(\frac{1}{r} - \frac{\varepsilon}{r-1}\right)n$. Thus, the positive co-degree thresholds for perfect and almost perfect matchings are ``smoothly'' related; this is in contrast to the co-degree thresholds for perfect and almost-perfect matchings.

The remainder of this paper is organized as follows. In Subsection~\ref{defs}, we summarize a few additional definitions and notation which will be used throughout the paper. In Section~\ref{matching section}, we give our results on perfect matchings. In Section~\ref{berge section}, we prove Theorem~\ref{Berge HC}. In Section~\ref{loose cycle section}, we prove Theorem~\ref{looseHC}. The proofs in Sections~\ref{matching section} and~\ref{berge section} use classical techniques; for the proof of Theorem~\ref{looseHC}, we shall use the absorbing method. Since the absorbing method is required only in Section~\ref{loose cycle section}, we postpone all discussion of the absorbing method and associated lemmas until then. Finally, in Section~\ref{further questions}, we briefly comment on our results and a variety of open directions in this area.

\subsection{Definitions and Notation}\label{defs}

Related to the varying notions of degree in hypergraphs, we shall require a few definitions and pieces of notation. In an $r$-graph $H$, the \textit{neighborhood} of a vertex $v \in V(H)$ is 
\[ N(v) := \{ u \in V(H)\setminus\{v\}: \{u,v\} \subseteq e \text{ for some } e \in E(H) \}, \]
while for an $(r-1)$-set $S \subset V(H)$, the \textit{co-degree neighborhood} of $S$ is
\[ N(S)  := \{u \in V(H)\setminus S: S\cup \{u\} \in E(H)\}. \]
Thus, $d_1(v) = |N(v)|$ and $d_{r-1}(S) = |N(S)|$. When $S$ is an explicitly defined set, we may drop set notation when denoting $N(S)$; for instance, if $S = \{x,y\}$, we write $N(x,y)$ rather than $N(\{x,y\})$ for the co-degree neighborhood of $S$.

Given $v \in V(H)$, the \textit{link graph} $L(v)$ is the $(r-1)$-uniform graph on vertex set $V(H) \setminus \{v\}$, where an $(r-1)$-set of vertices $S$ spans a hyperedge if and only if $S \cup \{v\} \in E(H)$. That is, the hyperedges of $L(v)$ correspond to those sets $S$ such that $v \in N(S)$. The study of link graphs is often useful in positive co-degree problems, particularly when $r = 3$ and thus the link graphs are actually $2$-graphs, to which we can apply any number of graph theoretic results. Our arguments shall often consider link graphs, though often we shall want to consider multiple link graphs simultaneously. In particular, given an $r$-graph $H$, the \textit{shadow graph} $S(H)$ of $H$ is the $(r-1)$-graph with vertex set $V(H)$ and edge set 
\[ E(S(H)) = \{ e : e \in E(L(v)) \text{ for some } v \in V(H)\}.\]
Related to the shadow graph, given an $(r-1)$-graph $K$, we say that $T \subset V(H)$ spans a \textit{co-degree K} if the subhypergraph of $S(H)$ induced on $T$ contains a copy of $K$. When working with minimum positive co-degree conditions, knowledge of the co-degree ``subgraphs'' of $H$ will prove useful in inferring the structure of $H$ itself.

\section{Perfect matchings}\label{matching section}

We begin by considering perfect matchings in $3$-graphs. Using Construction~\ref{main construction}, we observe that $H_{U,V}^3$ with $|V| = \lfloor \frac{n}{3} + 1 \rfloor$ does not contain a perfect matching and has minimum positive co-degree $\lceil \frac{2n}{3} \rceil - 2$. Our first theorem shows that the positive co-degree threshold for perfect matchings is exactly $\frac{2n}{3} - 1$. Note also that taking $H_{U,V}^3$ with $|V| = \left(\frac{1}{3} + \varepsilon \right)n$ yields a construction with minimum positive co-degree approximately $\left( \frac{2}{3} - \varepsilon \right)n$ and no matching of size greater than $\left( \frac{1 - \varepsilon}{3}\right)n$. Thus, the threshold for almost perfect matchings in $3$-graphs does not differ sharply from the threshold given by Theorem~\ref{threePM}.





\threePM*


\begin{proof}
The result is immediate for $n=3$, and is straightforward to check for $n = 6$. 
For $n\geq 9$, let $M$ be a maximum matching in $H$. If $M$ is a perfect matching, then we are done, so assume that $M$ does not cover all the vertices of $H$. Set $U := V(H)\setminus V(M)$, and $|U| =: k \geq 3$. Thus, $|V(M)| = n-k \leq n-3$. Our goal is to obtain a contradiction by showing that if $M$ is not perfect, then $M$ is not maximum, i.e., we can find a matching matching $M'$ strictly larger than $M$. We shall achieve this conclusion in several steps, which will involve an examination of the vertices in $U$.

First, observe that since $M$ is maximum, $U$ must be independent. We claim that, moreover, $U$ does not contain two disjoint pairs, say $x,y$ and $v,w$, such that $d_2(x,y) > 0$ and $d_2(v,w) > 0$ (here and hereafter, we abuse the formal notation, as $x,y$ is not a set). Indeed, suppose that such pairs exist in $U$. Since $U$ is independent, we have $N(x,y) \subseteq V(M)$ and $N(v,w) \subseteq V(M)$. By the bound on $\delta_2^+(H)$, both of these neighborhoods have size at least $\frac{2n}{3} - 1 > \frac{2(n-k)}{3}$. 
Thus, there must be some $3$-edge $abc \in M$ with $a \in N(x,y), b \in N(v,w)$. However, this yields a contradiction to the maximality of $M$, since $M' = M \setminus \{abc\} \cup \{axy, bvw\}$ is a strictly larger matching than $M$.

Next, we claim that $U$ contains at most two vertices whose neighborhoods intersect $U$ in at most one vertex. To the contrary, suppose $x,y,z \in U$ and each has at most one neighbor in $U$. 

We consider an auxiliary, mixed-uniformity, edge-colored multigraph $A$, with $V(A) = V(M)$. The edge set $E(A)$ will consist of the $3$-edges of $M$, together with the following colored $2$-edges: $C_x \subset E(A)$ is the set of edges 
$\{uv: u,v \in V(A) \text{ and } uvx \in E(H)\}.$
We assign color $x$ to every edge of $C_x$. Analogously, we define $C_y \subset E(A)$ to the set of edges 
$\{uv: u,v \in V(A) \text{ and } uvy \in E(H)\},$
each of which is assigned color $y$, and $C_z \subset E(A)$ is the set of edges 
$\{uv: u,v \in V(A) \text{ and } uvz \in E(H)\},$
each of which is assigned color $z$. Note that, if $N(u,v)$ contains more than one of $x,y,z$, we shall include multiedges between $u$ and $v$ of the appropriate colors.

Now, suppose that $x_1y_1z_1, x_2y_2z_2$ are $3$-edges of $M$ such that $x_1x_2 \in C_x$, $y_1y_2 \in C_y$, and $z_1z_2 \in C_z$. We shall call such a subgraph of $A$ an \textit{M-extender} for $x,y,z$.

\begin{figure}[h]
\begin{center}
\begin{tikzpicture}



    \draw[fill=pink,opacity=0.5]
(0,1) to[bend left=320] (0,0) to[bend left=320] (0,-1) to[bend left = 100] cycle
;

    \draw[fill=pink,opacity=0.5]
(3,1) to[bend right=320] (3,0) to[bend right=320] (3,-1) to[bend right = 100] cycle
;

\draw[red] (0,1) -- (3,1) node[pos= 0.5, above]{$x$};
\draw[blue] (0,0) -- (3,0) node[pos= 0.5, above]{$y$};
\draw[teal] (0,-1) -- (3,-1) node[pos= 0.5, above]{$z$};

\filldraw (0,0) circle (0.05 cm);
\filldraw (0,1) circle (0.05 cm);
\filldraw (0,-1) circle (0.05 cm);
\draw (0,1) node[above left]{$x_1$};
\draw (0,0) node[above right]{$y_1$};
\draw (0,-1) node[below left]{$z_1$};

\filldraw (3,0) circle (0.05 cm);
\filldraw (3,1) circle (0.05 cm);
\filldraw (3,-1) circle (0.05 cm);
\draw (3,1) node[above right]{$x_2$};
\draw (3,0) node[above left]{$y_2$};
\draw (3,-1) node[below right]{$z_2$};

\end{tikzpicture}

\caption{An $M$-extender}
\end{center}

\end{figure}
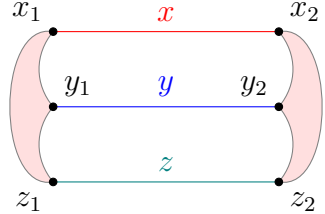

Now observe, if we find an $M$ extender for $x,y,z$ in $A$, then 
\[M' = (M \setminus \{x_1y_1z_1, x_2y_2z_2\}) \cup \{xx_1x_2, yy_1y_2, zz_1z_2\}\] 
is a matching of $H$ which is strictly larger than $M$. Thus, our goal is to find such a subgraph of $A$. We begin by finding one candidate $3$-edge whose vertices are appropriately incident to edges colored $x,y,$ and $z$.

First, we denote by $V_x$ the set of vertices in $V(A)$ which are incident to some edge in $C_x$, and analogously define $V_y, V_z$. For $e \in M$, we define 
\[ f(e) := \sum_{v \in \{x,y,z\}} |V(e) \cap V_v|. \]
Note that $ \sum_{e \in M} f(e)$ is simply $|V_x| + |V_y| + |V_z|$. Now, since $H$ contains no isolated vertices and $U$ is independent, there is some $v \in V(A)$ which has positive co-degree with $x$, i.e., $v \in V_x$. Since $\delta_2^+(H) \geq \frac{2n}{3} - 1$, there are at least $\frac{2n}{3} - 1$ vertices $u$ with $xvu \in E(H)$. By the assumption on $x$, at most one of these is contained in $U$, so 
\[ |V_x| \geq |\{v\} \cup \left( N(x,v) \cap V(A) \right)| \geq \frac{2n}{3} - 1 \](and analogously we bound $|V_y|, |V_z|$). Thus
\[ \sum_{e \in M} f(e) \geq 3\left( \frac{2n}{3} - 1\right) > 3\left( \frac{2(n-k)}{3}\right)\]
since $k \geq 3$.
By averaging, there is some $e \in M$ with $f(e) \geq 7$. Thus, every vertex of $e$ is in at least one of $V_x, V_y, V_z$, and in fact, either two vertices of $e$ are in all three sets, or one vertex of $e$ is in all $3$ and the other two vertices of $e$ are in two of the three sets. In either case, there is a labeling $x', y', z'$ of $V(e)$ such that $x' \in V_x$, $y' \in V_y$, and $z' \in V_z$. 

Set $M^- := M \setminus \{x'y'z'\}$. We define
$$V_x^- := \{v: v \in V(M^-) \text{ and } vx' \in C_x\}$$
and similarly define $V_y^-$ and $V_z^-$. Note that $|V_x^-| \geq \frac{2n}{3} - 4 = \frac{2(n-6)}{3}$, since $N(x,x') \geq \frac{2n}{3} - 1$ and at most three members of $N(x,x')$ are in $U \cup \{x', y', z'\}$.
Similarly, $|V_y^-| \geq \frac{2(n-6)}{3}$ and $|V_z| \geq \frac{2(n-6)}{3}$. Now, for $e \in M^-$ define 
\[ f^-(e) := \sum_{v \in \{x,y,z\}} |V_v^- \cap V(e)|. \]
Again, we have
\[ \sum_{e \in M^-} f^-(e) = \sum_{v \in \{x,y,z\}} |V_v^-|  \geq 3 \cdot \frac{2(n-6)}{3} \geq 2|V(M^-)|  \]
Thus, the average value of $f^-(e)$ is at least $6$. If some $e \in M^-$ has $f^-(e) > 6$, then $e$ and $x'y'z'$ must form an $M$-extender. If there is no $e \in M^-$ with $f^-(e) > 6$, then the above inequalities are tight, i.e., $|V(M^-)| = n-6$  and $|V_x^-| = |V_y^-| = |V_z^-| = \frac{2(n-6)}{3}$. In order to satisfy $\delta_2^+(H) \geq \frac{2n}{3} - 1$, we must have $y', z' \in N(x,x')$ and some $v_x \in I \cap N(x,x')$, and analogously for $N(y,y')$ and $N(z,z')$. In fact, since we assume that each of $x,y,z$ has at most one neighbor in $U$, and we have already shown that $U$ cannot contain two disjoint pairs of neighbors, we must have $v_x = v_y = v_z =: v$. Observe, $v$ cannot be equal to $x,y,$ or $z$, so $|U| \geq 4$, which implies $|V(M^-)| < n-6$, contradicting the tightness of the above inequalities.

Thus, we may assume that $U$ does not contain $3$ vertices whose neighborhoods intersect $U$ in at most one vertex. Coupled with the condition that $U$ does not contain two disjoint pairs of neighbors, this implies that $|U| \leq 3$, and that the neighborhood of some vertex of $U$ intersects $U$ in two vertices. Thus, if $U$ is not empty, then $U = \{x,y,z\}$ and $d_2(x,y), d_2(x,z)$ are strictly positive. 

We shall again define an auxiliary multigraph $A$ as above and find an $M$-extender for $x,y,z$. Again, our first goal will be to find a $3$-edge $e$ which admits a labelling of its vertices as $x',y',z'$ such that $x' \in V_x, y' \in V_y, z' \in V_z$. Now, since $x,y,z$ may have two neighbors in $U$, our initial bound is that $|V_x|, |V_y|,$ and $|V_z|$ are at least $\frac{2n}{3} -2 = \frac{2(n-3)}{3}$. However, observe that since $d_2(x,y) \geq \frac{2n}{3} - 1$ and $N(x,y)$ does not include $z$, we in fact have $|V_x| \geq \frac{2n}{3} - 1$ and $|V_y| \geq \frac{2n}{3} - 1$. Analogously, examining $N(x,z)$ shows that $|V_z| \geq \frac{2n}{3} - 1$ 
Thus, we will again have by averaging that there is some $e \in M$ with $f(e) \geq 7$, and there is a labeling $x', y', z'$ of the vertices of edge $e$ such that $x' \in V_x$, $y' \in V_y$, and $z' \in V_z$.

We define $M^-, V_x^-, V_y^-, V_z^-,$ and $f^-$ as before. Let $U^+ = \{x,y,z, x',y',z'\}$. Again, 
\[\sum_{e \in M^-} f^-(e) = |V_x^-| + |V_y^-| + |V_z^-|,\]
and if $\sum_{e \in M^-} f^-(e) > 3 \cdot \frac{2(n-6)}{3}$, then there exists $e \in M^-$ with $f^-(e) > 6$ which forms an $M$-extender with $x'y'z'$. Thus, we will be done if we show that $|V_x^-| + |V_y^-| + |V_z^-| >2(n-6)$. Since $\delta_2^+(H) \geq \frac{2n}{3} - 1$, we have 
\[\sum_{v \in \{x,y,z\}} |N(v,v')| \geq 2n - 3,\]
so we will be done if 
\[\sum_{v \in \{x,y,z\}} |N(v,v') \cap U^+| < 9,\]
since $V_v^- = N(v,v') \setminus U^+$. 

Now, for each $v \in \{x,y,z\}$, we have $|N(v,v') \cap U^+| \leq 4$, since $N(v,v')$ does not contain either $v$ or $v'$. We first show that if any $|N(v,v') \cap U^+|$ is equal to $4$, then the desired average holds.

Suppose without loss of generality that $|N(x,x') \cap U^+| = 4$. Then $N(x,x')$ contains all four of $y,y',z,z'$. If some $3$-edge of $H$ is spanned by $y,y',z,z'$, then $U^+$ contains two vertex disjoint $3$-edges. These $3$-edges can be added to $M^-$ to create a perfect matching of $H$, contradicting the maximality of $M$. Thus, if $|N(x,x') \cap U^+| = 4$, then $|N(y,y') \cap U^+| \leq 2$ and $|N(z,z') \cap U^+| \leq 2$. The sum of the three is thus less than $9$, as desired.

Hence, we may assume that none of the three intersections has size greater than $3$. Thus, if one of the three intersections has size at most $2$, we are done. We now show that this must occur.

Suppose to the contrary all intersections are of size $3$. Thus, without loss of generality, $N(x,x')$ contains $z,z'$, and one of $y,y'$. Suppose (again, without loss of generality) that $y \in N(x,x')$. Now, if $y' \in N(z,z')$, then $xx'y$ and $zz'y'$ form a matching in $U^+$ which extends $M^-$ to a perfect matching. This cannot occur, so we must have $N(z,z') \cap U^+ = \{x,x',y\}$. Note that $x,x', z,z'$ induce a clique. Now, consider $N(y,y')$. Any $3$-edge in $U^+$ containing $y,y'$ can now be paired with an edge from the clique spanned by $x,x',z,z'$ to form a matching of $U^+$, again extending $M^-$ to a perfect matching.
\end{proof}

Next, we consider perfect matchings in $r$-graphs. Using Construction~\ref{main construction} again, we observe that $H_{U,V}^r$ with $|V| = \lfloor \frac{n}{r} + 1 \rfloor$ does not contain a perfect matching and has minimum positive co-degree $\lceil \frac{(r-1)n}{r} \rceil - r + 1$. We conjecture that the general positive co-degree threshold for perfect matchings in $r$-graphs is $\delta_{r-1}^+(H) = \left(\frac{r-1}{r}\right) n - r + 2$. Note that for $r = 3$, this is the threshold given by Theorem \ref{threePM}. The structural analysis required to obtain the exact result of Theorem \ref{threePM} does not seem easy to generalize; in particular, as $r$ grows, there are many more possible configurations of positive co-degree sets among the unmatched vertices to consider. However, at the cost of a slightly worse bound, we are able to avoid such case analysis and obtain the following theorem, which is best possible up to the additive constant. We remark that the bound on $n$ is best possible in this regime, as no non-empty $r$-graph on fewer vertices will satisfy this minimum positive co-degree constraint.

\rPM*

\begin{proof}

    Let $H$ be any $r$-graph satisfying the above conditions and let $M$ be a matching of maximum size in $H$. Define $U$ to be the vertices of $H$ not contained in any edge of $M$, i.e., $U=V(H)\setminus V(M)$. Our goal is to show that no vertices are contained in $U$.

    Towards a contradiction, suppose that $U$ is nonempty. We define the following function $f$ on pairs $(S,e)$ of subsets $S\subseteq U$ and edges $e\in H$:
    \[
    f(S,e) = \left|\{v\in e |  \text{ $v$ is contained in an edge with $S$}\}\right|.
    \]
    With this in hand, we present the following lemma.

    \begin{lemma}\label{matchinglemma}
        Let $M$ be a matching in an $r$-uniform hypergraph $H$ and let $U$ be the vertices unmatched by $M$. Let $S_1,S_2,\ldots, S_k$ be disjoint sets in $U$ each of size $\ell$, $r>\ell>0$ such that $k \leq r$ and for each $S_i$, $1\leq i \leq k$, $\sum_{e\in M}f(S_i,e) \geq (\frac{r-1}{r})n$. Then there exists an edge $e\in M$ containing distinct vertices $v_1,v_2,\ldots, v_k$ such that $S_i\cup\{v_i\}$ is contained in an edge of $H$ for all $1\leq i \leq k$.
    \end{lemma}
    \begin{proof}
        Fix $k\leq r$ and let $S_1,S_2,\ldots, S_k$ be such a collection of sets. Then we have
        \[
        \sum_{i=1}^k\sum_{e\in M}f(S_i,e) \geq k\left(\frac{r-1}{r}\right)n.\] As there are fewer than $\frac{n}{r}$ edges in $M$, at least one edge $e:=v_1v_2 \dots v_r$ contributes strictly more than $k(r-1)$ to the sum. Construct an auxiliary bipartite graph with one class, $X$, being the sets $S_1,S_2,\ldots, S_k$, and the other class, $Y$, being the vertices of this edge $e$, with edges whenever $S_i\cup\{v_j\}$ is contained in an edge of $H$. We show that this graph contains an $X$-saturating matching, i.e., a matching in for which each vertex of $X$ is in an edge:

        We see first that this auxiliary graph contains more than $k(r-1)$ edges, so no vertex in $Y$ is isolated. Since $k\leq r$, we have $k(r-1)\geq (k-1)r$, implying no vertex of $X$ is isolated either. Let $W$ be any subset of $X$. The above two facts show that if $|W|=1$ or $|W|=k$, $|W|\leq |N(W)|$. Now, suppose $W$ is not of these two sizes, and assume by way of contradiction that $|W|=i>|N(W)|$. Then $|N(W)|\leq i-1$, so the number of edges incident to $W$ is at most $i(i-1)=i^2-i$. There are at most $(k-i)r$ incident to vertices in $X\setminus W$. Our auxiliary graph has more than $k(r-1)$ edges, and one can see that
        \[
        k(r-1)-(i^2-i)-(k-i)r=kr-k-i^2+i-kr+ir = i(r+1)-i^2-k=i(r+1-i)-k.
        \]
        Because $k\leq r,$ $i(r+1-i)-k\geq i(r+1-i)-r=(i-1)(r-i)$. Since we have assumed $i>1$ and $i<k \leq r$, both $i-1$ and $r-i$ are positive. So $(i-1)(r-i)>0$. This gives us $(i^2-i)+(k-i)r<k(r-1)$, a contradiction.
    \end{proof}

    We now prove the following claim:
    \begin{claim}
        There do not exist $r$ vertices in $U$ such that each is contained in an edge of $H$.
    \end{claim}
    \begin{proof}[Proof of Claim]
        Towards a contradiction, assume that $v_1, v_2,\ldots, v_r$ are distinct vertices in $U$ such that each is contained in an edge of $H$. Define $\mathcal{S}=\{S_1,\cup\dots\cup S_r\}$, where $S_1=\{v_1\},\dots, S_{r}=\{v_r\}$. Greedily add vertices one-by-one from $U\setminus \mathcal{S}$ to $S_1$ such that at each point, $\{v\}\cup S_1$ is contained in an edge of $H$. Repeat this for $S_2,\ldots, S_r$. One can easily check that, by our co-degree condition, each $S_i$ is in an edge with at least $(\frac{r-1}{r})n+r^2$ other vertices. At most $r^2$ vertices are in $\mathcal{S}$, and no $S_i$ is in an edge with a vertex in $U\setminus \mathcal{S}$. Combining this with our assumption, each $S_i$ is in an edge with at least $(\frac{r-1}{r})n$ vertices in edges of $M$. Indeed, each $S_i$ satisfies $\sum_{e\in M}f(S_i,e) \geq (\frac{r-1}{r})n$. We define the following procedure:

        Start with $M'=M, U'=U$. While each $S_i$ is of size less than $r$, take the $S_i$'s in $\mathcal{S}$ of least size, i.e., those of minimal cardinality. Without loss of generality, suppose these sets are $S_1,S_2,\ldots, S_j$. Applying Lemma~\ref{matchinglemma}, we can find an edge $e\in M'$ containing distinct vertices $v_1,v_2,\ldots, v_j$ such that $S_i\cup\{v_i\}$ is contained in an edge of $H$ for all $1\leq i \leq j$. Remove edge $e$ from $M'$ and replace $S_i$ with $S_i\cup\{v_i\}$ for each of $S_1,S_2,\ldots, S_j$. Add to $U'$  the vertices of $e$. Now, greedily add vertices from $U'\setminus \mathcal{S}$ to $S_1$ such that $S_1\cup\{v\}$ is in an edge of $H$, and then repeat for $S_2,\ldots, S_r$. After doing so, no $S_i$ is in an edge with a vertex in $U\setminus \mathcal{S}$, and there are at most $r^2$ vertices in $\mathcal{S}$. So each $S_i$ is in an edge with at least $(\frac{r-1}r)$ vertices of edges in $M$.

        This process clearly terminates in at most $r-1$ steps. Once this process terminates, $\{S_1,S_2, \ldots ,S_{r}\}$ will be a set of $r$ many edges of $H$ that are disjoint from each other, and disjoint from all edges of $M'$. $M'$ will have lost at most $r-1$ edges while running the procedure, so $M'$ combined with our sets $S_1,S_2,\ldots, S_{r}$ is a matching with one more edge than $M$, a contradiction.
    \end{proof}
    The above claim yields the theorem: observe that, as the number of vertices in $H$ is divisble by $r$, and $H$ has no isolated vertices, the claim implies $M$ is a perfect matching.
\end{proof}

\section{Berge Hamiltonian cycles}\label{berge section}

We now consider Berge Hamiltonian cycles. Firstly, we claim that taking Construction~\ref{main construction} with $|V| = \lceil \frac{n + 1}{2} \rceil$ yields an example of a $3$-graph $H$ with $\delta_2^+(H) = \lfloor \frac{n + 1}{2} \rfloor - r + 2$ and no Berge Hamiltonian cycle. Suppose for a contradiction that $H$ contains some Berge Hamiltonian cycle $C$, and let $v_1, v_2, \dots, v_n, v_1$ be the associated cyclic ordering of $V(H)$. Since $|V| > \frac{n}{2}$, any cyclic ordering of the vertices of $H$ must include an adjacent pair of vertices from $V$, say $v_i$ and $v_{i+1}$. Since $C$ is a Berge Hamiltonian cycle, there exists a $3$-edge $e_i$ of $H$ with $v_i, v_{i+1} \in e_i$, a contradiction, since $V$ is strongly independent.

Our next theorem, restated from the introduction, shows that Construction~\ref{main construction} is again tight.

\bergetheorem*

\begin{proof}
We shall in fact prove a somewhat stronger statement: namely, that if $H$ satisfies the hypotheses of the theorem, then $H$ contains a Berge Hamiltonian cycle
\[C = v_1e_1v_2e_2\dots v_{n-1}e_{n-1}v_ne_nv_1\]
such that for any $i \neq j$, $e_{i}$ contains at most one of $v_j, v_{j+1}$. We proceed by induction on $r$. 
Suppose $r=2$. Then $n\ge 3$ and $\delta_{r-1}^+(H)  = \delta(H) \ge \frac{n}{2}$. By Dirac's Theorem, we have a (Berge) Hamiltonian cycle which vacuously satisfies the condition on each $e_i$. 

Now suppose $r > 2$ and that the theorem holds for $(r-1)$-graphs. Let $H$ be an $r$-graph satisfying the conditions of the theorem statement.
Let $H':=S(H)$ be the shadow graph of $H$.
Observe, since $H$ contains no isolated vertices, neither does $H'$. Also observe that $\delta_{r-1}^+(H) \geq \frac{n}{2} - r + 2$ implies $\delta_{r-2}^+(H') \geq \frac{n}{2} - (r-1) + 2$. Thus, by the inductive hypothesis, $H'$ contains a Berge Hamiltonian cycle, say
\[C' = v_1 e_1' v_2 \dots v_n e_n' v_1.\]
We wish to translate $C'$ into a Berge Hamiltonian cycle $C$ of $H$. We do so by replacing each $e_i'$ with an $r$-edge $e_i$ containing the vertices of $e_i'$.

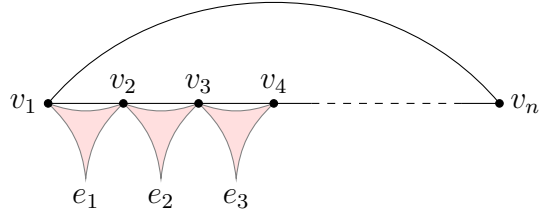
\begin{figure}[h!]
\begin{center}
\begin{tikzpicture}


    \draw[fill=pink,opacity=0.5]
(0,0) to[bend right=20] (1,0) to[bend right=20] (0.5,-1) to[bend right = 20] cycle
;

    \draw[fill=pink,opacity=0.5]
(1,0) to[bend right=20] (2,0) to[bend right=20] (1.5,-1) to[bend right = 20] cycle
;

    \draw[fill=pink,opacity=0.5]
(2,0) to[bend right=20] (3,0) to[bend right=20] (2.5,-1) to[bend right = 20] cycle
;



\draw (0.5, -1) node[below]{$e_1$};
\draw (1.5, -1) node[below]{$e_2$};
\draw (2.5, -1) node[below]{$e_3$};

\filldraw (0,0) circle (0.05 cm);
\draw (0,0) node[left]{$v_1$};

\filldraw (1,0) circle (0.05 cm);
\draw (1,0) node[above]{$v_2$};

\filldraw (2,0) circle (0.05 cm);
\draw (2,0) node[above]{$v_3$};

\filldraw (3,0) circle (0.05 cm);
\draw (3,0) node[above]{$v_4$};

\filldraw (6,0) circle (0.05 cm);
\draw (6,0) node[right]{$v_n$};

\draw[dashed] (3.5,0) -- (5.5,0); 

\draw (0,0) -- (3.5,0);

\draw (5.5,0) -- (6,0);

\draw (0,0) to[bend left = 50] (6,0);

\end{tikzpicture}
\caption{Extending from an $r=2$ Hamiltonian cycle to  an $r=3$ Berge Hamiltonian cycle}
\end{center}
\end{figure}

We must make these replacements in such a way that, for $i \neq j$, $e_i$ contains at most one of $v_j, v_{j+1}$. Note that the vertices of $e_i'$ indicate at most $2(r-3) + 2 = 2r - 4$ vertices with which we may not extend $e_i'$: namely, $e_i$ must not contain $v_{i-1}$ or $v_{i+2}$, and, if $v_j \in e_i'$, then $e_i$ must not contain $v_{j-1}$ or $v_{j+1}$. Using the bound on $n$ and the condition on $\delta_{r-1}^+(H)$, we know that $\delta_{r-1}^+(H) \geq 2r -3$, so each $e_i'$ can be extended into an $r$-edge $e_i$ which satisfies the desired condition. Note that the condition that $e_i$ contains at most one of $v_j,v_{j+1}$ immediately implies that $e_i \neq e_j$ for all $i \neq j$. Thus, 
\[C = v_1e_1v_2\dots v_ne_nv_1\]
is a Berge Hamiltonian cycle in $H$.
\end{proof}

\section{Loose Hamiltonian cycles}\label{loose cycle section}

Finally, we consider loose Hamiltonian cycles in $3$-graphs. Observe that a loose Hamiltonian $3$-cycle yields a cyclic vertex ordering with the property that any two vertices adjacent in the order have positive co-degree. Thus, Construction~\ref{main construction} with $|V| = \lceil \frac{n + 1}{2} \rceil$ again provides an example of a $3$-graph with $\delta_2^+(H)$ approximately $\frac{n}{2}$ and no loose Hamiltonian cycle. 

While we do not obtain an exactly matching upper bound, the main theorem of this section, restated from the introduction, shows that the above example is asymptotically best possible.

\loosethm*

We shall prove Theorem \ref{looseHC} via the absorbing method; for similar applications, see \cite{BHScycle,handissertation,Rodl_2017,Reiher_2019_mindegree},  Before beginning the proof, we briefly summarize the method and the lemmas which we shall need. 

We define a \textit{($3$-uniform) loose path} of length $k$ as a set of $k$ $3$-edges 
\[v_1v_2v_3,v_3v_4v_5, \dots, v_{2k-1}v_{2k}v_{2k+1}\]
such that $v_i\neq v_j$ for all $i\neq j$. Since we work only with $3$-uniform hypergraphs in this section, we shall for conciseness refer to a $3$-uniform loose path (of length $k$) simply as a ($k$-vertex) loose path.

The absorbing method works roughly as follows. Given an $n$-vertex $3$-graph $H$, with $n$ sufficiently large and $\delta_2^+(H) \geq \left( \frac{1}{2} + \varepsilon \right)n$, we wish to find three things: an \textit{absorber}, a \textit{reservoir}, and an \textit{almost path-tiling} of bounded size. The absorber is a relatively short loose path $P$ with the property that, given any set $Q \subset V(H)$ which is relatively small and has the appropriate parity, there exists a loose path on $V(P) \cup Q$ which has the same endpoints as $P$. The reservoir is a relatively small set of vertices with a particular connecting property: given a small number of pairs $a_i,b_i$ of vertices in $H$, we can find pairwise disjoint sets of vertices in the reservoir which connect $a_i$ and $b_i$ via a loose path. Finally, the almost path-tiling will be a set of boundedly many loose paths which are pairwise disjoint and span all but some small proportion of the vertices in $H$. 

Once we have shown that we can find an absorber, a reservoir, and an almost path-tiling, we proceed as follows. Using the fact that our absorber and reservoir can be made very small, we can actually guarantee the existence of an absorber, a reservoir, and an almost path-tiling which are pairwise disjoint. We use the reservoir to connect the endpoints of the resulting set of boundedly many, pairwise disjoint loose paths, thus creating an almost-spanning loose cycle. Finally, any uncovered vertices can be absorbed to create a loose Hamiltonian cycle.

Due to the large number of auxiliary techniques and results required, we further divide this section into subsections, as follows. In Subsection~\ref{absorber subsection}, we state and prove our absorbing and reservoir lemmas, which largely follow from work of Buß, Han, and Schacht in~\cite{BHScycle}. In Subsection~\ref{tiling subsection}, we prove our almost path-tiling lemma, adapting the approaches of Buß, Han, and Schacht in~\cite{BHScycle} and Han in~\cite{handissertation} by making use of an auxiliary tiling lemma and some hypergraph regularity techniques which we translate to the positive co-degree setting. In Subsection~\ref{main thm subsection}, we combine the lemmas from Subsections~\ref{absorber subsection}
and ~\ref{tiling subsection} to prove Theorem~\ref{looseHC}.

\subsection{Absorbing and reservoir lemmas}\label{absorber subsection}

We begin by stating two lemmas: a connecting lemma, which will be needed for finding our absorber, and our reservoir lemma. These lemmas are presented without proof, because they directly follow from~\cite{BHScycle}. The statements of both lemmas are shown to hold for $3$-graphs with minimum vertex degree at least $\left(\frac{1}{4} + \varepsilon\right) \binom{n}{2}$; note that a $3$-graph with minimum positive co-degree $\left( \frac{1}{2} + \varepsilon \right)n$ and no isolated vertices satisfies this vertex degree bound.

We say that a set of triples $(x_i,y_i,z_i)_{i\in [k]}$ \textit{connects} $(a_i,b_i)_{i\in [k]}$ if
\begin{itemize}
    \item $\bigcup_{i\in [k]}|\{a_i,b_i,x_i,y_i,z_i\}|=5k$, i.e., the pairs and triples are all disjoint, and
    \item for all $i\in [k]$ we have $\{a_i,x_i,y_i\},\{y_i,z_i,b_i\}\in H$
\end{itemize}

\begin{figure}[h]
\begin{center}
\begin{tikzpicture}[scale = 1.3]



    \draw[fill=pink,opacity=0.5]
(0,0) to[bend right=50] (1,0) to[bend right=50] (2,0) to[bend left = 50] cycle
;

    \draw[fill=pink,opacity=0.5]
(-2,0) to[bend left=50] (-1,0) to[bend left=50] (0,0) to[bend right = 50] cycle
;

\filldraw (0,0) circle (0.05 cm);
\filldraw (1,0) circle (0.05 cm);
\filldraw (-1,0) circle (0.05 cm);
\draw (1,0) node[above right]{$z$};
\draw (0,0) node[above right]{$y$};
\draw (-1,0) node[below left]{$x$};

\filldraw (2,0) circle (0.05 cm);
\filldraw (-2,0) circle (0.05 cm);
\draw (2,0) node[above right]{$b$};
\draw (-2,0) node[above left]{$a$};

\end{tikzpicture}

\caption{$(x,y,z)$ connects $(a,b)$}
\end{center}
\end{figure}
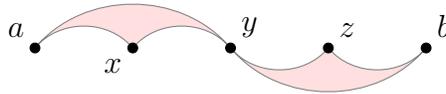

\begin{restatable}[Buß-Han-Schacht~\cite{BHScycle}]{lemma}{connecting}\label{connecting}
Let $\gamma>0$, let $m\geq 1$ be an integer, and let $H=(V,E)$ be a $3$-uniform hypergraph on $n$ vertices with $\delta_1(H)\geq (\frac{1}{4}+\gamma)\binom{n}{2}$ and $m\leq \gamma n /12$. For every set $(a_i,b_i)_{i\in [m]}$ of mutually disjoint pairs of distinct vertices, there exists a set of triples $(x_i,y_i,z_i)_{i\in [m]}$ connecting $(a_i,b_i)_{i\in [m]}$
\end{restatable}

\begin{restatable}[Buß-Han-Schacht~\cite{BHScycle}]{lemma}{reservoir}\label{reservoir}
For all $0<\gamma<1/4$ there exists an $n_0$ such that for every $3$-uniform hypergraph $H=(V,E)$ on $n>n_0$ vertices with minimum vertex degree $\delta_1(H)\geq (\frac{1}{4}+\gamma)\binom{n}{2}$ there is a set $R$ of size at most $\gamma n$ with the following property: For every system $(a_i,b_i)_{i\in [k]}$ consisting of $k\leq \gamma^3n/12$ mutually disjoint pairs of vertices from $V$ there is a triple system connecting $(a_i,b_i)_{i\in [k]}$ which, moreover, contains vertices from $R$ only.
\end{restatable}

Next, we shall use Lemma~\ref{connecting} to prove an absorbing lemma. 

\begin{lemma}\label{absorbing}
    For all $0  <\varepsilon < \frac{1}{2}$ , there exist $\beta > 0$ and $n_0 \in \mathbb{N}$ such that the following holds. Let $H$ be an $n$-vertex $3$-graph with $n \geq n_0$ and $\delta_2^+(H) \geq \left( \frac{1}{2} + \varepsilon \right)n$. Then there is a loose path $P$ with $|V(P)| \leq \frac{5 \varepsilon^3 n}{6}$ such that for all subsets $U \subset V(H) \setminus V(P)$ with $|U| \leq \beta n$ and $|U| \in 2\mathbb{N}$, there exists a loose path $Q$ in $H$ such that $V(Q) = V(P) \cup U$ and $P$ and $Q$ have the same endpoints.
\end{lemma}

\begin{proof}

We begin by considering pairs of vertices. Let $x,y \in V(H)$ and $P = v_1 \dots v_7$ be a $7$-vertex loose path with $V(P) \in V(H)\setminus\{x,y\}$. We say that $P$ \textit{absorbs} $\{x,y\}$ if $v_2xv_4, v_4yv_6 \in E(H)$. In particular, if $P$ absorbs $\{x,y\}$, then
\[Q = v_1v_3v_2xv_4yv_6v_5v_7\]
is a loose path with $V(Q) = V(P) \cup \{x,y\}$ and the same endpoints as $P$. We first claim that for any pair $\{x,y\}$, there are many $7$-vertex loose paths which absorb $\{x,y\}$.

\begin{claim}\label{smallabsorbers}
Suppose $n$ is large enough that $\varepsilon^2 n^2 \geq \frac{1}{2}n$ and $\varepsilon n \geq 8$. For every pair $x,y \in V(H)$, there exist at least $\frac{\varepsilon^3 n^7}{16}$ $7$-vertex loose paths which absorb $\{x,y\}$.

\end{claim}
\begin{proof}[Proof of Claim \ref{smallabsorbers}]

We provide a procedure for building $3$-edge loose paths to absorb $\{x,y\}$. Firstly, choose $v_4 \in N(x) \cap N(y)$. Since $|N(x)|$ and $|N(y)|$ are each at least $\left( \frac{1}{2} + \varepsilon \right)n$, there are at least $2 \varepsilon n$ choices for $v_4$. Next, choose $v_3, v_5$ such that $v_3v_4v_5 \in E(H)$, and $v_3, v_5$ are not equal to $x$ or $y$. There are at least $\binom{|N(v_4)|}{2}$ $3$-edges which contain $v_4$, at most $2n$ of which contain either $x$ or $y$. Using that $|N(v_4)| \geq \left(\frac{1}{2} + \varepsilon \right)n$, we have at least $\frac{n^2}{8}$ choices for the $3$-edge $v_3v_4v_5$ by the bound on $n$. 

We next select $v_2, v_6$. We shall require $v_2,v_6$ to satisfy the following:

\begin{enumerate}

\item $d_2(v_2,v_3) > 0$ and $d_2(v_5,v_6) > 0$;

\item $v_2v_4x \in E(H)$ and $v_4v_6y \in E(H)$. 

\item $v_2,v_6 \not\in \{x,y,v_3,v_4,v_5\}$.

\end{enumerate}

We have selected $v_4$ so that $d_2(x,v_4) > 0$ and $d_2(y,v_4) > 0$, so we have $|N(x,v_4)| \geq \left( \frac{1}{2} + \varepsilon \right)n$ and $|N(y,v_4)| \geq \left( \frac{1}{2} + \varepsilon \right)n$. We also have that $|N(v_3)|$ and $|N(v_5)|$ are at least $\left( \frac{1}{2} + \varepsilon \right)n$, so 
\[|N(x,v_4) \cap N(v_3)| \geq 2\varepsilon n\]
and
\[|N(y,v_4) \cap N(v_5)| \geq 2 \varepsilon n\]
The only other restrictions on the selection of $v_2$ and $v_6$ are that they are not equal to any vertex already selected, or to each other. Thus, there are at least $\varepsilon n$ possible selections for each of $v_2$ and $v_6$. 

Finally, we must select $v_1$ and $v_7$. We simply require $v_1 \in N(v_2,v_3)$ and $v_7 \in N(v_5,v_6)$, and that $v_1, v_7$ are distinct vertices which are not equal to any previously specified vertices. There are at least $\frac{n}{2}$ ways to thus select each of $v_1,v_7$ (since $\varepsilon n \geq 8$).

Thus, in total, there are at least 
\[2\varepsilon n \cdot \frac{n^2}{8} \cdot (\varepsilon n)^2 \cdot \left(\frac{n}{2} \right)^2 = \frac{\varepsilon^3 n^7}{16} \]
3-edge loose paths which absorb $\{x,y\}$.
\end{proof}


In similar fashion to~\cite{BHScycle}
we construct a family $\F$ of $7$-tuples $(v_1\ldots v_7)$ from among all $7$-vertex loose paths in $H$ by independently selecting each tuple with probability $p=\varepsilon^3n^{-6}/12544$. Using standard probabilistic arguments, it is straightforward to show that, for $n$ large enough, with non-zero probability the family satisfies 
\begin{enumerate}
    \item $|\F| \le \varepsilon^3 n /12$;
    \item for every pair $x,y\in V(H)$, there exist $\frac{p\varepsilon^3 n^7}{32}$ tuples $(v_1,\ldots, v_7)\in \F$ such that $v_1\ldots v_7$ is a $7$-vertex loose path that absorbs $\{x,y\}$;
    \item the number of intersecting pairs of 7-tuples in $\F$ is at most $\frac{p\varepsilon^3 n^7}{64}$.
\end{enumerate}

We eliminate one path from each intersecting pair. There remain $k\le \varepsilon^3n/12$ loose paths $(v_1^i, \dots, v_7^i)_{i\in [k]}$. Moreover, since we deleted at most $\frac{p\varepsilon^3 n^7}{64}$ $7$-tuples, any pair $\{x,y\}$ disjoint from these remaining paths must be absorbed by at least $\frac{p\varepsilon^3 n^7}{64}$ of them. 

Now, our goal is to apply Lemma~\ref{connecting} on the vertex pairs $(v_7^i,v_1^{i+1})_{i\in [k-1]}$ to create a loose path with ends $v_1^1$ and $v_7^k$. However, we wish to ensure that our connecting vertices do not themselves come from these $7$-vertex loose paths. To this end, arbitrarily pair the remaining vertices $v_{j}^i$, $j\not\in \{1,7\}$ with each other, and observe that we still have $7k/2<\varepsilon n /12$ pairs (using the upper bound on $\varepsilon$). Thus we may apply Lemma~\ref{connecting} such that the connecting vertices are all disjoint from our $7$-vertex loose paths, yielding a loose path containing at most $\frac{7\varepsilon^3 n}{12} + \frac{3 \varepsilon^3 n}{12} = \frac{5 \varepsilon^3 n}{6}$ vertices that can absorb an arbitrary set of least $\frac{p\varepsilon^3 n^7}{32}=\beta n$ vertices.
\end{proof}

\subsection{Tiling lemmas}\label{tiling subsection}

Given a fixed family of hypergraphs $\mathcal{K}$ and an $n$-vertex hypergraph $H$, a \textit{partial} $\mathcal{K}$-\textit{tiling} of $H$ is a family $\mathcal{T}$ of vertex-disjoint subhypergraphs of $H$, each of which is isomorphic to some member of $\mathcal{K}$. We say that such a family $\mathcal{T}$ is a $\mathcal{K}$\textit{-tiling} of $H$ if $\mathcal{T}$ spans $V(H)$, i.e., each vertex of $H$ is in a member of $\mathcal{T}$. We say that $\mathcal{T}$ is an $\alpha$\textit{-deficient} $\mathcal{K}$-tiling of $H$ if $\mathcal{T}$ leaves at most $\alpha n$ vertices of $H$ uncovered.

The final component which we require is an almost path-tiling result. Formally, let $\mathcal{P}$ be the family of $3$-graphs consisting of all loose paths. The goal of this subsection is to prove the following lemma.

\begin{lemma}\label{path tiling lemma}

For any fixed $\varepsilon > 0, \alpha > 0$, there exist integers $n_0, p$ such that the following holds for $n > n_0$. If $H$ is an $n$-vertex $3$-graph with $\delta_2^+(H) \geq \left( \frac{1}{2} + \varepsilon \right)n$, then $H$ admits an $\alpha$-deficient $\mathcal{P}$-tiling $\mathcal{T}$ with $|\mathcal{T}| \leq p$.
 
\end{lemma}

This path-tiling result will be the most involved lemma in this section to prove, and our strategy will involve the use of the weak hypergraph regularity lemma. Before proving any tiling results, we begin with a brief discussion of hypergraph regularity. Szemer\'edi's regularity lemma is known to generalize to hypergraphs, and in fact can be considered to generalize in a number of different ways. The ``weak'' hypergraph regularity lemma is a straightforward generalization of  Szemer\'edi's regularity lemma, which shall nonetheless be powerful enough for our purposes. For more on hypergraph regularity lemmas, see for instance \cite{Rodl_regularity,tao2005variant,gowers2007hypergraph}. Since we only work with $3$-graphs, we shall state the weak regularity lemma for $3$-graphs only. We begin with some notation. 

Let $H$ be a $3$-graph and $V_1, V_2, V_3$ be pairwise disjoint, nonempty subsets of $V(H)$. We denote by $e(V_1,V_2,V_3)$ the number of $3$-edges with precisely one vertex in $V_i$ for each $i \in \{1,2,3\}$. We define the \textit{density} of the triple $(V_1,V_2,V_3)$ to be
\[d(V_1,V_2,V_3) = \frac{e(V_1,V_2,V_3)}{|V_1||V_2||V_3|}.\]
Given $\varepsilon >0$, $d \geq 0$, we say that the triple $(V_1,V_2,V_3)$ is $(\varepsilon,d)$-\textit{regular} if
\[|d(V_1', V_2', V_3') - d| \leq \varepsilon\]
for any triple $(V_1', V_2',V_3')$ such that for each $i \in \{1,2,3\}$, we have $V_i' \subseteq V_i$ and $|V_i'| \geq \varepsilon |V_i|$. Given $\varepsilon > 0$, we call the triple $(V_1,V_2,V_3)$ $\varepsilon$-\textit{regular} if there is
some $d \geq 0$ for which $(V_1,V_2,V_3)$ is $(\varepsilon,d)$-regular. Using this notation, the following lemma holds.

\begin{lemma}[Weak hypergraph regularity]\label{weakreg}
Let $\varepsilon > 0$ and $t_0 \in \mathbb{N}$ be given. There exist $T_0, n_0 \in \mathbb{N}$ such that any $3$-graph $H$ on $n \geq n_0$ vertices admits a vertex partition 
\[V_0 \sqcup V_1 \sqcup \dots \sqcup V_t \]
satisfying the following properties:

\begin{enumerate}
\item $t_0 \leq t\leq T_0$;
\item $|V_1| = \dots = |V_t|$, and $|V_0| \leq \varepsilon n$;
\item All but at most $\varepsilon \binom{t}{3}$ of the triples chosen from $\{V_1, \dots ,V_t\}$ are $\varepsilon$-regular.

\end{enumerate}

\end{lemma}

We call a partition as guaranteed in Lemma~\ref{weakreg} an $\varepsilon$-\textit{regular partition}. A related notion is the \textit{cluster hypergraph}. Given an $\varepsilon$-regular partition $V_0 \sqcup V_1 \sqcup \dots \sqcup V_t$ of $H$ and a fixed $d \geq 0$, the associated $(\varepsilon, d)$-\textit{cluster hypergraph} $K = K(\varepsilon,d)$ is the 3-graph with vertex set $\{1, \dots, t\}$ and $ijk \in E(K)$ if and only if $(V_i,V_j,V_k)$ is $\varepsilon$-regular with $d(V_i,V_j,V_k) \geq d$.

The above discussion will be relevant as follows. Suppose $H$ is a large $3$-graph with $\delta_2^+(H) \geq \left(\frac{1}{2} + \varepsilon \right)n$. Our ultimate goal is to find an almost path-tiling of $H$. The rough idea is to first apply weak hypergraph regularity to $H$, then almost tile the resulting cluster graph $K$ with some small structure, using the fact (proven below) that $K$ nearly inherits $\delta_2^+(H)$. The resulting almost tiling of $K$ can then be lifted to an almost path-tiling of $H$, with each small tile in $K$ indicating a pair of loose paths which nearly span the corresponding clusters.  

To find the desired almost tiling of $K$, we shall require some knowledge of $\delta_2^+(K)$, which we now investigate. We start by stating a ``clean-up lemma'', proved by Halfpap, Lemons, and Palmer~\cite{pcddensity}.

\begin{lemma} \label{cleanup lemma}

Let $H$ be an $n$-vertex $r$-graph and fix $0< \varepsilon <1$ small enough that 
\[(r+1)! \sqrt[2^{r-1}]{\varepsilon}n^r < |E(H)|.\] Let $H_1$ be a subhypergraph of $H$ obtained by the deletion of at most $\varepsilon n^r$ $r$-edges. Then $H_1$ has a subhypergraph $H_2$ 
with $\delta_{r-1}^+(H_2) \geq \delta_{r-1}^+(H) - 2^r r! \sqrt[2^{r-1}]{\varepsilon}n$. Moreover, $H_2$ can be obtained from $H$ by the deletion of at most $(r+1)! \sqrt[2^{r-1}]{\varepsilon}n^r$ $r$-edges.

\end{lemma}

We shall use Lemma~\ref{cleanup lemma} to prove the following lemma, which makes precise the claim that a cluster graph $K$ of $H$ ``nearly inherits'' $\delta_2^+(H)$.

\begin{lemma}\label{codeg reg}

Let $0 < \varepsilon < 1/2$ and $t_0 \in \mathbb{N}$ be given, and fix $n_0, T_0$ as in Lemma~\ref{weakreg}. Let $H$ be an $n$-vertex $3$-graph with $\delta_2^+(H) \geq cn$ and $n \geq n_0$. Suppose that 
\[V_0 \sqcup V_1 \sqcup \dots \sqcup V_t\]
is a partition of $V(H)$ satisfying the properties stated in Lemma~\ref{weakreg}. Fix $d \geq 0$, and let $K = K(\varepsilon,d)$ be the $(\varepsilon,d)$-cluster hypergraph associated with the given partition. Set  
\[\alpha = 2 \varepsilon + \frac{1}{t} + d\]
and 
\[f(\alpha) = 48\sqrt[4]{\alpha}.\]
If $|E(H)| > 12\sqrt[4]{\alpha} n^3$, then $K$ contains a subhypergraph $K'$ with $\delta_2^+(K') \geq (c - f(\alpha))t$. Moreover, if $H$ contains $g(n)$ isolated vertices, then $K'$ contains at most 
\[\left(\frac{g(n)}{n} + \frac{288 \sqrt[4]{\alpha}}{c^2}\right) t\]
isolated vertices.

\end{lemma}

\begin{proof}

We begin by deleting from $H$ all $3$-edges $e$ satisfying one (or more) of the following conditions:

\begin{enumerate}[(1)]

\item $e$ intersects $V_0$;

\item $e$ intersects some $V_i$ more than once;

\item $e$ spans a triple of classes $(V_i,V_j,V_k)$ which is not $\varepsilon$-regular;

\item $e$ spans a triple of classes $(V_i,V_j,V_k)$ with $d(V_i,V_j,V_k) < d$. 

\end{enumerate}

Let $H_1$ be the subhypergraph of $H$ obtained by the described deletions. Observe that vertices $i,j,k$ in $K$ span a $3$-edge if and only if there exists a $3$-edge in $H_1$ which spans the associated classes $V_i,V_j,V_k$. 

We first estimate the number of $3$-edges which must be deleted from $H$ to obtain $H_1$. Since $|V_0| \leq \varepsilon n$, there are at most 
\[\varepsilon n \binom{n}{2} < \varepsilon n^3\]
$3$-edges satisfying condition (1). For a fixed class $V_i$ with $i > 0$, there are at most 
\[\binom{|V_i|}{3} + n \binom{|V_i|}{2} < 2n \binom{n/t}{2} < \frac{n^3}{t^2}\]
$3$-edges which intersect $V_i$ in more than one vertex. There are $t$ classes $V_i$ with $i > 0$, so fewer than $\frac{n^3}{t}$ $3$-edges which satisfy condition (2). There are at most $\varepsilon \binom{t}{3}$ triples $(V_i,V_j,V_k)$ that are not $\varepsilon$-regular, each spanning at most $\left( \frac{n}{t}\right)^3$ $3$-edges. So at most 
\[ \varepsilon \binom{t}{3} \left(\frac{n}{t} \right)^3 < \varepsilon n^3\]
$3$-edges satisfy condition (3). Finally, a triple $(V_i,V_j,V_k)$ with $d(V_i, V_j, V_k) < d$ spans fewer than $d \left( \frac{n}{t} \right)^3$ $3$-edges, so there are at most
\[\binom{t}{3} d \left( \frac{n}{t}\right)^3 < dn^3\]
$3$-edges satisfying condition (4). Thus, in total, we delete fewer than
\[ \left( 2 \varepsilon + \frac{1}{t} + d\right)n^3 = \alpha n^3\]
$3$-edges to obtain $H_1$ from $H$. 

Now, we apply Lemma~\ref{cleanup lemma} to conclude that $H_1$ has a subgraph $H_2$ with $\delta_2^+(H_2) \geq (c - f(\alpha))n$. Now, consider the $3$-graph $K'$ on vertex set $\{1,2, \dots, t\}$, with $ijk \in E(K')$ if and only if $E(H_2)$ contains a $3$-edge spanning clusters $V_i,V_j,V_k$. Observe that $K'$ is a subhypergraph of $K$. We claim that $\delta_2^+(K') \geq (c - f(\alpha))t$. Indeed, suppose $i,j \in V(K')$ have positive co-degree. Then there exists some pair of vertices $v_i \in V_i, v_j \in V_j$ in $H_2$ such that $v_i,v_j$ have positive co-degree. Thus, $d(v_i,v_j) \geq (c - f(\alpha))n$. By the construction of $H_2$, none of these neighbors are in $V_0,V_i$, or $V_j$. All classes $V_k$ with $k \neq 0$ have size $n/t$, so $d(v_i,v_j) \geq (c -f(\alpha))n$ implies that $v_i,v_j$ have positive co-degree with vertices in at least $(c - f(\alpha))t$ classes $V_k$ with $k \not\in \{0,i,j\}$. Thus, in $K'$, vertices $i,j$ have at least $(c - f(\alpha))t$ neighbors. 

Finally, we estimate the number of isolated vertices in $K'$. First, since $\delta_2^+(H) \geq c n$, observe that any vertex of $H$ is either isolated or is contained in at least $\binom{cn + 1}{2}$ $3$-edges of $H$. By Lemma \ref{cleanup lemma}, we obtain $H_2$ by deleting at most $24 \sqrt[4]{\alpha} n^3$ $3$-edges from $H$, so at most 
\[3 \cdot \frac{24 \sqrt[4]{\alpha}n^3}{\binom{cn + 1}{2}} = \frac{144 \sqrt[4]{\alpha}n^3}{(cn + 1)cn} < \frac{144 \sqrt[4]{\alpha}n}{c^2}\]
vertices of $H$ become isolated after the deletions required to obtain $H_2$. Thus, $H_2$ contains at most $g(n) + \frac{144 \sqrt[4]{\alpha}n}{c^2}$ isolated vertices in total. An isolated vertex in $K'$ corresponds to a class $V_i$ in $H_2$ such that every vertex in $V_i$ is isolated. Since each class $V_i$ contains at most $(1-\varepsilon)n/t > n/2t$ vertices (by choice of $\varepsilon$), at most 
\[\left(\frac{g(n)}{n} + \frac{144 \sqrt[4]{\alpha}}{c^2}\right) 2t\]
classes of $H_2$ can consist solely of isolated vertices. Thus, $K'$ contains at most \[\left(\frac{2g(n)}{n} + \frac{288 \sqrt[4]{\alpha}}{c^2}\right) t\] isolated vertices.
\end{proof}

Given $H$ with $\delta_2^+(H) \geq \left(\frac{1}{2} + \varepsilon \right)$, our goal is to find a helpful almost-tiling of a cluster hypergraph $K$ of $H$. Actually, we shall find an almost-tiling of the subhypergraph $K'$ of $K$ given by Lemma~\ref{codeg reg}. Observe that if $H$ has no isolated vertices, then Lemma~\ref{codeg reg} guarantees that $K'$ will have few isolated vertices (since $\alpha$ will be taken very small). In particular,  if we can almost tile the non-isolated vertices of $K'$, then we shall in fact almost tile all of $K$. The utility of this auxiliary almost-tiling of $K$ is shown by the following lemma, proved in~\cite{handissertation}.

\begin{lemma}[Han~\cite{handissertation}]\label{reg path lemma}
Fix $\varepsilon > 0, d > 2 \varepsilon$, and let $m$ be an integer with $m \geq \frac{d}{\varepsilon(d - 2 \varepsilon)}$. Suppose that $(V_1, V_2, V_3)$ is $(\varepsilon, d)$-regular, with $|V_1| = |V_3| = m$ and $|V_2| = 2m$. Then there is a loose path $P$ which spans all but at most $8\varepsilon m/d + 3$ vertices in $V_1 \sqcup V_2 \sqcup V_3$.

\end{lemma}

Using Lemma~\ref{reg path lemma}, we will be able to find our desired almost path-tiling in $H$ so long as we can partition $V(H)$ into regular triples of the desired sizes. We shall describe how to find this partition using an auxiliary almost-tiling of $K$ by a useful $3$-graph; we then ``lift'' the tiles in $K$ to indicate disjoint regular triples in $H$. Denote by $C_4^3$ the (unique) $3$-graph on four vertices and two $3$-edges. The following lemma will imply that $K$ admits an $\alpha$-deficient $C_4^3$-tiling for any $\alpha > 0$, so long as $|V(K)|$ is sufficiently large.

\begin{lemma}\label{C tiling}

Fix $\varepsilon > 0$, $\alpha > 0$, and let $n$ be sufficiently large. Let $H$ be an $n$-vertex $3$-graph with $\delta_2^+(H) \geq \left( \frac{1}{2} + \varepsilon \right)n$ and no isolated vertices.
Then $H$ contains an $\alpha$-deficient $C_4^3$-tiling.

\end{lemma}

\begin{proof}
Let $\mathcal{C}$ be a largest $C_4^3$-tiling of $H$. Let $T$ be the set of vertices which are covered by $\mathcal{C}$, and $S$ the set of uncovered vertices. If $|S| \leq \alpha n$, we are done, so suppose for a contradiction that $|S| > \alpha n$. Our goal is to exhibit a partial $C_4^3$-tiling $\mathcal{C}'$ of $H$ with $|\mathcal{C}'| > |\mathcal{C}|$. Towards this aim, we begin by investigating the structure of $S$. Our first goal is to argue that $S$ contains very few positive co-degree pairs. We investigate \textit{co-degree cherries}, which are cherries in the shadow graph of $H$.

Consider the two configurations illustrated in Figure \ref{cherryconfigs}, which we shall term \textit{switching configurations}. In each instance, we take a co-degree cherry in $S$ and a copy $C \in \mathcal{C}$ of $C_4^3$. Either illustrated configuration allows us to ``switch'' $C$ for a new $C_4^3$ copy which uses one or two vertices from $C$ and two or three vertices from $S$. The described switch creates a new partial $C_4^3$-tiling of equal size to $\mathcal{C}$, which is not itself a contradiction. However, if a single $C \in \mathcal{C}$ admits two switching configurations which are vertex disjoint, then we can replace $C$ with two new copies of $C_4^3$ to obtain a larger partial $C_4^3$-tiling than $\mathcal{C}$. 

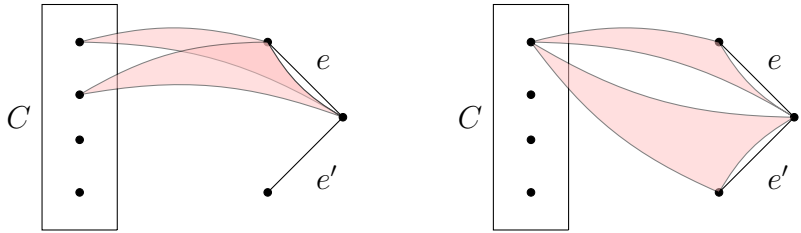
\begin{figure}[h]
\centering

\begin{tikzpicture}
    \filldraw (0,0) circle (0.05 cm);
    \filldraw (-1,1) circle (0.05 cm);
    \filldraw (-1,-1) circle (0.05 cm);
    \draw (0,0) -- (-1,1) node[pos = 0.5, above right]{$e$};
    \draw (0,0) -- (-1,-1) node[pos = 0.5, below right]{$e'$};

    \draw (-3,-1.5) -- (-4,-1.5) -- (-4, 1.5) -- (-3, 1.5) -- (-3,-1.5);
    \draw (-4,0) node[left]{$C$};
    \filldraw (-3.5, -1) circle (0.05 cm);
    \filldraw (-3.5, 1) circle (0.05 cm);
    \filldraw (-3.5, -0.3) circle (0.05 cm);
    \filldraw (-3.5, 0.3) circle (0.05 cm);

    \draw[fill=pink,opacity=0.5]
(-1,1) to[bend right=15] (0,0) to[bend right=15] (-3.5,1) to[bend left=15] cycle
;
    \draw[fill=pink,opacity=0.5]
(-1,1) to[bend right=15] (0,0) to[bend right=15] (-3.5,0.3) to[bend left=15] cycle
;

    \filldraw (6,0) circle (0.05 cm);
    \filldraw (5,1) circle (0.05 cm);
    \filldraw (5,-1) circle (0.05 cm);
    \draw (6,0) -- (5,1) node[pos = 0.5, above right]{$e$};
    \draw (6,0) -- (5,-1) node[pos = 0.5, below right]{$e'$};

    \draw (3,-1.5) -- (2,-1.5) -- (2, 1.5) -- (3, 1.5) -- (3,-1.5);
    \draw (2,0) node[left]{$C$};
    \filldraw (2.5, -1) circle (0.05 cm);
    \filldraw (2.5, 1) circle (0.05 cm);
    \filldraw (2.5, -0.3) circle (0.05 cm);
    \filldraw (2.5, 0.3) circle (0.05 cm);

    \draw[fill=pink,opacity=0.5]
(5,1) to[bend right=15] (6,0) to[bend right=15] (2.5,1) to[bend left=15] cycle
;

    \draw[fill=pink,opacity=0.5]
(5,-1) to[bend left=15] (6,0) to[bend left=15] (2.5,1) to[bend right=15] cycle
;

\end{tikzpicture}

\caption{Two switching configurations}\label{cherryconfigs}

\end{figure}

Using this idea, we can show that $S$ must contain $O(n)$ positive co-degree pairs. 

\begin{claim}
If $S$ contains three pairwise disjoint co-degree cherries, then $\mathcal{C}$ is not maximal.

\end{claim}

\begin{proof}[Proof of Claim.]

Suppose for the sake of a contradiction that $S$ contains three pairwise disjoint co-degree cherries. For each edge $e_i$ contained in these cherries, and each $C_j \in \mathcal{C}$, let $f(e_i, C_j)$ be the number of vertices in $C_i$ contained in the co-degree neighborhood of $e_i$'s endpoints. Observe, for any positive co-degree pair $x,y \in S$, we have $|N(x,y) \cap S| \leq 1$, else $S$ contains a copy of $C_4^3$ and we can immediately extend $\mathcal{C}$. Thus, for each of the $e_i$'s, we have 
\[ \sum_{j} f(e_i, C_j) \geq \left( \frac{1}{2} + \varepsilon \right)n - 1,\]
and so 
\[ \sum_{i,j} f(e_i, C_j) \geq \left( 3 + 6\varepsilon \right)n - 6 \geq \left( 3 + \varepsilon \right)n,\]
using that $n$ is sufficiently large.
Since $\mathcal{C}$ contains $\frac{|T|}{4} \leq \frac{(1 - \alpha)n}{4}$ copies of $C_4^3$, by averaging we have that some copy $C$ of $C_4^3$ in $T$ has
\[\sum_{i} f(e_i, C) \geq \frac{(3 + \varepsilon)n}{\frac{(1 - \alpha)n}{4}} = \frac{4(3 + \varepsilon)}{1 - \alpha} > 12.\]

Now, choose $C$ with $\sum_{i} f(e_i, C) \geq 13$. We shall use the cherries and this $C$ to form two disjoint $C_4^3$ copies. Since some case analysis is required to find these copies, we begin by labeling vertices and co-degree edges. Say the vertices in $C$ are $v_1, v_2, v_3, v_4$. Since we will not use the internal structure of $C$ at all, it does not matter which of the $v_j$'s are spanned by two $3$-edges of $C$, and which by one. We shall label the co-degree edges in the first cherry as $e_1, e_2$, the edges in the second as $e_3, e_4$, and the edges in the third as $e_5,e_6$. For ease of reference, we illustrate the situation in Figure \ref{baseconfig}.

\begin{figure}
    \centering
\begin{tikzpicture}

\draw (0,-3) -- (4,-3) -- (4,-2) -- (0,-2) -- (0,-3);

\draw (0,-2.5) node[left]{$C$};

\filldraw (0.5,-2.5) circle (0.05 cm);  
\filldraw (1.5,-2.5) circle (0.05 cm);  
\filldraw (2.5,-2.5) circle (0.05 cm);  
\filldraw (3.5,-2.5) circle (0.05 cm); 

\draw (0.5, -2.5) node[below]{$v_1$};
\draw (1.5, -2.5) node[below]{$v_2$};
\draw (2.5, -2.5) node[below]{$v_3$};
\draw (3.5, -2.5) node[below]{$v_4$};

\draw (-0.5,-1) -- (0, 0.5) -- (0.5,-1);

\draw (-0.5,0) node{$e_1$};
\draw (0.5,0) node{$e_2$};

\draw (1.5,-1) -- (2, 0.5) -- (2.5,-1);

\draw (1.5,0) node{$e_3$};
\draw (2.5,0) node{$e_4$};

\draw (3.5,-1) -- (4, 0.5) -- (4.5,-1);
\draw (3.5,0) node{$e_5$};
\draw (4.5,0) node{$e_6$};

\filldraw (-0.5,-1) circle (0.05 cm);
\filldraw (0.5,-1) circle (0.05 cm);
\filldraw (0, 0.5) circle (0.05 cm);

\filldraw (1.5,-1) circle (0.05 cm);
\filldraw (2.5,-1) circle (0.05 cm);
\filldraw (2, 0.5) circle (0.05 cm);

\filldraw (3.5,-1) circle (0.05 cm);
\filldraw (4.5,-1) circle (0.05 cm);
\filldraw (4, 0.5) circle (0.05 cm);

\end{tikzpicture}
    \caption{Three co-degree cherries in $S$ and $C \in \mathcal{C}$} \label{baseconfig}
   
\end{figure}
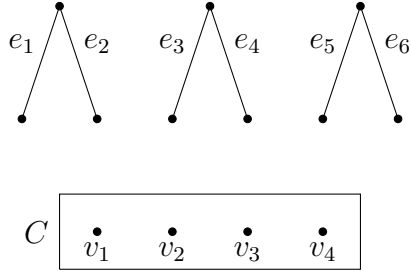

For ease of notation and terminology, we shall perform our analysis on the auxiliary labeled bipartite graph $G(V,U)$, where $V = \{v_1,\dots,v_4\}$, $U = \{e_1,\dots,e_6\}$, and $e_iv_j$ is an edge of $G$ if $v_j$ and the endpoints of $e_i$ span a $3$-edge of $H$. Thus, $\sum_i d(e_i) = \sum_{i} f(e_i, C) \geq 13$. 

We first observe that if some vertex in $U$ has degree $4$, then $\mathcal{C}$ is not maximal. Indeed, without loss of generality, suppose $d(e_1) = 4$. Since $\sum_i d(e_i) \geq 13$, one of $e_3,e_4,e_5,e_6$ has degree at least $2$. Without loss of generality, assume $v_1,v_2$ are neighbors of $e_3$ in $G$. Since $d(e_1) = 4$, we have $v_3, v_4 \in N(e_1)$. We have thus found two vertex disjoint switching configurations: $e_1$ paired with $v_3, v_4$ and $e_3$ paired with $v_1,v_2$. 

Now, suppose no vertex in $U$ has degree $4$. Consider the sums $d(e_1) + d(e_2)$, $d(e_3) + d(e_4)$, and $d(e_5) + d(e_6)$. Since $\sum_i d(e_i) \geq 13$, one of these sums must be at least $5$. Without loss of generality, $d(e_1) + d(e_2) \geq 5$ and $d(e_1) = 3$. We now consider two cases, determined by the interaction of $N(e_1)$ and $N(e_2)$.

\noindent \textbf{Case 1: $e_1$ and $e_2$ have two common neighbors.}

Without loss of generality, say $v_1, v_2 \in N(e_1) \cap N(e_2)$. Now, $d(e_1) + d(e_2) \leq 6$, so either $d(e_3) + d(e_4)$ or $d(e_5) + d(e_6)$ is at least 4. Without loss of generality,  $d(e_3) + d(e_4) \geq 4$, and we may assume $d(e_3) \geq 2$. Now, if $N(e_3) \neq \{v_1, v_2\}$, we can find a pair of vertex disjoint switching configurations as follows: $e_3$ is paired with $\{v_i,v_j\} \in N(e_3)$ such that $\{v_i, v_j\} \neq \{v_1,v_2\}$, and $e_1, e_2$ are paired with $v_k \in \{v_1, v_2\}$ such that $v_k \not\in \{v_i, v_j\}$. Thus, we may assume that $N(e_3) = \{v_1,v_2 \}$. So $d(e_3) = 2$, implying that $d(e_4) = 2$, and we analogously have that $N(e_4) = \{v_1, v_2\}$. We now find two vertex disjoint switching configurations: $e_1, e_2$ are paired with $v_1$ and $e_3, e_4$ are paired with $v_2$. 

In any event, $\mathcal{C}$ is not maximal.

\noindent \textbf{Case 2: $e_1$ and $e_2$ do not have two common neighbors.}

Since $d(e_1) + d(e_2) \geq 5$ and $d(e_1) = 3$, the arrangement is (without loss of generality) as follows: $N(e_1) = \{v_1,v_2,v_3\}$ and $N(e_2) = \{v_3,v_4\}$. Now, consider $e_i$ with $i \geq 3$. Observe, if $\{v_1, v_2\} \subseteq N(e_i)$, then we can find two vertex disjoint switching configurations: $e_i$ is paired with $v_1, v_2$ and $e_1,e_2$ are paired with $v_3$. Furthermore, if $j \neq 4$ and $\{v_j, v_4\} \subseteq N(e_i)$, then we can find two vertex disjoint switching configurations: $e_i$ is paired with $\{v_j, v_4\}$, and $e_1$ is paired with $\{v_1,v_2,v_3\}\setminus \{v_j\}$. Thus, we may assume that each $e_i$ with $i \geq 3$ has degree at most $2$; moreover, if its degree is exactly $2$, then $N(e_i)$ is either equal to $\{v_1,v_3\}$ or $\{v_2,v_3\}$. Since $d(e_1) + d(e_2) = 5$ and $\sum_i d(e_i) \geq 13$, every $e_i$ with $i \geq 3$ must have $d(e_i) = 2$. In particular, $v_3 \in N(e_i)$ for every $e_i$. We can now find two vertex disjoint switching configurations: $e_3$ and $e_4$ are paired with $v_3$, and $e_1$ is paired with $v_1, v_2$.  

In any event, $\mathcal{C}$ is not maximal.
\end{proof}

The following claim will now suffice to show that $\mathcal{C}$ is not maximal. 
\begin{claim}
Suppose $S$ does not contain three pairwise disjoint co-degree cherries. Then there exists a partial $C_4^3$-tiling $\mathcal{C'}$ of $H$ with the following properties:
\begin{enumerate}
\item $|\mathcal{C}| = |\mathcal{C}'|$;

\item Among the vertices left uncovered by $\mathcal{C'}$, there exist three pairwise disjoint co-degree cherries.

\end{enumerate}

\end{claim}

\begin{proof}[Proof of Claim]

Suppose $x,y \in T$ have at least $2$ neighbors in $S$, say $u,v$. Then $x,y,u,v$ span a copy of $C_4^3$. Our strategy is to use the sparsity of $S$ to find pairs of vertices in $T$ with many co-degree neighbors in $S$. These pairs will indicate copies of $C_4^3$ which we will ``switch'' with copies in $\mathcal{C}$ in order to create a new partial tiling $\mathcal{C'}$. To ensure that $\mathcal{C'}$ satisfies the conditions of the claim, the pairs which we select from $T$ must lie in a sufficiently nice configuration. We shall again use averaging to guarantee such a configuration. We start by estimating the number of pairs in $T$ which have a useful number of co-degree neighbors in $S$.

We shall say that a pair of vertices $x, y \in T$ is a \textit{high S-degree pair} if $|N(x,y) \cap S| \geq 18$. We say that $x \in T$ is a \textit{high S-degree vertex} if $|N(x) \cap S| \geq 18$. Observe that if $x,y$ is a high $S$-degree pair, then both $x$ and $y$ are high $S$-degree vertices. Our goal is to find $C, C' \in \mathcal{C}$ which span many high $S$-degree pairs. 

Since $S$ does not contain three pairwise disjoint co-degree cherries, we know that $S$ must be quite sparse in co-degree edges. Specifically, we observe that there are at most $8$ vertices in $S$ which have positive co-degree with more than $7$ vertices in $S$.
Thus, $S$ contains at most $\frac{1}{2} (8n + 7n) = \frac{15}{2} n$ co-degree edges. Now,

\[ \sum_{x \in C_i, y \in C_j : i \neq j} |N(x,y) \cap S|\]
is equal to the number of $3$-edges $x,y,z$ such that $x,y$ are in distinct copies of $C_4^3$ in $\mathcal{C}$ and $z \in S$. Observe, the only other edges which intersect $S$ must either use a positive co-degree pair in $S$, or use a positive co-degree pair spanned by some copy of $C_4^3$ in $\mathcal{C}$. There are at most $\frac{15}{2}n$ positive co-degree pairs in $S$, and at most $6|\mathcal{C}| \leq \frac{3}{2} n$  positive co-degree pairs spanned by the copies of $C_4^3$ in $\mathcal{C}$, so the above sum counts all but at most $9n^2$ $3$-edges which intersect $S$. Now, since $H$ contains no isolated vertices, we know that each vertex in $S$ is in at least $\left(\frac{1}{2} + \varepsilon \right)n$ positive co-degree pairs, and thus in at least $\binom{\left(\frac{1}{2} + \varepsilon \right)n}{2} = \left( \frac{1}{4} + \varepsilon' \right) \binom{n}{2}$ $3$-edges, for some $\varepsilon'>0$. Summing vertex degrees in $S$ double-counts only those $3$-edges which intersect $S$ in at least two vertices, of which there are at most $\frac{15}{2}n^2$. So

\[ \sum_{x \in C_i, y \in C_j : i \neq j} |N(x,y) \cap S| \geq \left( \frac{1}{4} + \varepsilon' \right) \binom{n}{2}|S| - \frac{33}{2}n^2.\]
Since $|S| \geq \alpha n$ and thus $\varepsilon' \binom{n}{2}|S| \geq 33 n^2$ for $n$ large enough, we have
\[\left( \frac{1}{4} + \varepsilon' \right) \binom{n}{2}|S| - \frac{33}{2}n^2 \geq \left( \frac{1}{4} + \varepsilon'' \right)\binom{n}{2}|S|\]
for some $\varepsilon'' > 0$.
Now, we wish to find the average number of high $S$-degree pairs between two copies $C_i, C_j$ of $C_4^3$ in $\mathcal{C}$. Let $h(C_i,C_j)$ denote the number of pairs $x \in V(C_i), y \in V(C_j)$ with $|N(x,y) \cap S| \geq 18$. Observe that 
\[\sum_{x \in C_i, y \in C_j} |N(x,y) \cap S| \leq h(C_i,C_j)|S| + (16 - h(C_i,C_j))17.\]
Thus,
\[\sum_{x \in C_i, y \in C_j : i \neq j} |N(x,y) \cap S| \leq \sum_{C_i \neq C_j} \left(  h(C_i, C_j)|S| + (16 - h(C_i,C_j))17 \right)  \]

Now, there are $\binom{|\mathcal{C}|}{2}$ pairs $C_i \neq C_j$. Since $|\mathcal{C}| = |T/4| \leq \frac{(1-\alpha)n}{4}$, this implies that the above sum is over at most 
\[ \left( \frac{1}{16} - \alpha' \right) \binom{n}{2}\]
pairs for some $\alpha' > 0$. Setting $\overline{h}$ equal to the average value of $h(C_i,C_j)$ and combining our bounds on 
\[\sum_{x \in C_i, y \in C_j : i \neq j} |N(x,y) \cap S|,\] we have

\[ \left( \frac{1}{4} + \varepsilon'' \right)\binom{n}{2}|S| \leq \overline{h}|S|\left(\frac{1}{16} - \alpha' \right) \binom{n}{2}  + 272 \left(\frac{1}{16} - \alpha' \right) \binom{n}{2}.\]
Now, dividing through by $\binom{n}{2} |S|$, using the fact that $|S| \geq \alpha n$, we have 
\[ \left( \frac{1}{4} + \varepsilon'' \right) \leq \overline{h}\left(\frac{1}{16} - \alpha' \right) + \frac{272 \left(\frac{1}{16} - \alpha' \right)}{\alpha n}\]
So, for some $0 < \varepsilon'''<\varepsilon''$, and $n$ large enough,
\[\left(\frac{1}{4} + \varepsilon''' \right) \leq \overline{h} \left( \frac{1}{16} - \alpha' \right),\]
i.e.
\[\frac{\left(\frac{1}{4} + \varepsilon''' \right)}{\left( \frac{1}{16} - \alpha' \right)} \leq \overline{h}.\]
We claim that this implies that there are three mutually disjoint pairs of $C_4^3$ copies with $h$ value at least 5. Indeed, consider the auxiliary graph $G$ whose vertices are the $C_4^3$ copies of $\mathcal{C}$, with edges connecting two copies if their $h$ value is at least $5$. Either we can find the desired mutually disjoint pairs, or $G$ does not contain a matching of size $3$. So by a theorem of Erd\H{o}s and Gallai \cite{erdosmatching}, $G$ contains at most $\binom{2}{2} + 2(|\mathcal{C}| - 1) < 2 |\mathcal{C}|$ edges (if $2|\mathcal{C}| \geq 10$, which we may assume). Thus,
\[\sum_{C_i \neq C_j} h(C_i,C_j) \leq 16 (2 |\mathcal{C}|) + 4 \binom{|\mathcal{C}|}{2}, \]
and so
\[\overline{h} \leq 32 \frac{|\mathcal{C}|}{\binom{|\mathcal{C}|}{2}} + 4  = 4 + \frac{64}{|C|-1}.\]
Note that if $|\mathcal{C}|$ grows with $n$, then for any $\beta >0$, $\overline{h} < 4 + \beta$ when $n$ is sufficiently large. (In fact, it is easy to see that $|\mathcal{C}|$ must be at least $\left(\frac{1}{2} + \varepsilon \right)\frac{n}{4} - \frac{1}{4}$, since otherwise there exists a positive co-degree pair $x,y$ outside of $\mathcal{C}$ and $|N(x,y)| \geq 4|\mathcal{C}| + 2$.) Thus, since 
\[4 < \frac{\left(\frac{1}{4} + \varepsilon''' \right)}{\left( \frac{1}{16} - \alpha' \right)},\]
we have 
\[\overline{h} < \frac{\left(\frac{1}{4} + \varepsilon''' \right)}{\left( \frac{1}{16} - \alpha' \right)}\]
for $n$ large enough, a contradiction.

Thus, without loss of generality, we have $h(C_1,C_2) \geq 5, h(C_3,C_4) \geq 5$, and $h(C_5,C_6) \geq 5$. Thus, in each of the three pairs, we can find the configuration depicted in Figure~\ref{tile matching}: a matching of size 2 between $C_i$ and $C_j$ indicating two high $S$-degree pairs, and a fifth vertex in one of $C_i, C_j$ which is in a high $S$-degree pair.

\begin{figure}[h]
\centering 

\begin{tikzpicture}

\draw (0,0) -- (4,0) -- (4,1) -- (0,1) -- (0,0);

\draw (0,0.5) node[left]{$C_i$};

\draw (0,-2) -- (4,-2) -- (4,-1) -- (0,-1) -- (0,-2);

\draw (0,-1.5) node[left]{$C_j$};   

\draw[red] (0.5,0.5) -- (0.5,-1.5);

\draw[red] (1.5,0.5) -- (1.5, -1.5);

\draw[dotted, thick, red] (2.5,0.5) circle (0.3 cm);

\filldraw (0.5,0.5) circle (0.05 cm);  
\filldraw (1.5,0.5) circle (0.05 cm);  
\filldraw (2.5,0.5) circle (0.05 cm);  
\filldraw (3.5,0.5) circle (0.05 cm);  

\filldraw (0.5,-1.5) circle (0.05 cm);  
\filldraw (1.5,-1.5) circle (0.05 cm);  
\filldraw (2.5,-1.5) circle (0.05 cm);  
\filldraw (3.5,-1.5) circle (0.05 cm);

\end{tikzpicture}
\caption{High $S$-degree pairs between $C_i$ and $C_j$; the circled vertex in $C_i$ is also a member of some high $S$-degree pair}
\label{tile matching}

\end{figure}
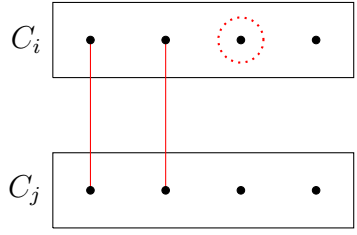

Each matching of size 2 indicates two pairs of vertices which can extend to several $C_4^3$ copies with both vertices in $S$. We first extend all matching edges to pairwise disjoint $C_4^3$ copies with two vertices in $S$. We now form a new partial tiling $\mathcal{C}'$ which contains $\mathcal{C} \setminus \{C_1,\dots,C_6\}$ and these six new copies of $C_4^3$. Let $T'$ be the vertices covered by $\mathcal{C}'$, and $S'$ the set of uncovered vertices. Observe, $S'$ now contains three vertices, say $v_1, v_2, v_3$, which were high $S$-degree vertices in $T$. Since $12$ vertices of $S$ have been moved to $T'$, each of $v_1, v_2, v_3$ thus has positive co-degree with at least $6$ vertices in $S \cap S'$. We can thus form three pairwise disjoint co-degree cherries in $S'$ with centers $v_1, v_2, v_3$.
\end{proof}

With the above claim established, we are done: either the set $S$ uncovered by our original partial tiling $\mathcal{C}$ contains three co-degree cherries, or else we can apply the claim to shift to an equally large partial tiling $\mathcal{C}'$ with an uncovered set $S'$ containing three co-degree cherries. In either case, the presence of three co-degree cherries in the uncovered set guarantees that we can extend to a strictly larger partial tiling, contradicting the maximality of $\mathcal{C}$.
\end{proof}

For the sake of completeness, we formally describe the process by which the above lemmas imply Lemma~\ref{path tiling lemma}.

\begin{proof}[Proof of Lemma~\ref{path tiling lemma}]

Fix $\varepsilon >0, \alpha > 0$. We choose $\varepsilon', t_0, d,\beta,$ and $n_0, T_0$ so that the following hold:

\begin{enumerate}

\item $1152 \sqrt[4]{2\varepsilon' + \frac{1}{t_0} + d} + \beta + \frac{3 \varepsilon'}{d} + \varepsilon' < \min\{\alpha, \varepsilon\}$;

\item $d > 2 \varepsilon'$;

\item $t_0$ is sufficiently large that Lemma~\ref{C tiling} implies that any $3$-graph $H$ on $t$ vertices with $t \geq \left(1 - 1152 \sqrt[4]{2\varepsilon' + \frac{1}{t_0} + d}\right)t_0$ and $\delta_2^+(H) \geq \left(\frac{1}{2} + \varepsilon - 48\sqrt[4]{2\varepsilon' + \frac{1}{t_0} + d}\right)t$ contains a $\beta$-deficient $C_4^3$-tiling;

\item Given $\varepsilon', t_0$, Lemma~\ref{weakreg} holds with $n_0, T_0$;

\item $n_0$ is large enough that $\frac{8 \varepsilon n_0}{2T_0d} + 3 < \frac{6 \varepsilon n_0}{T_0d}$.

\end{enumerate}

Now, let $H$ be a $3$-graph on $n \geq n_0$ vertices with $\delta_2^+(H) \geq \left( \frac{1}{2} + \varepsilon \right)n$. We apply the Weak Regularity Lemma (Lemma~\ref{weakreg}) with $\varepsilon', t_0$ to obtain a partition 
\[V(H) = V_0 \sqcup V_1 \sqcup \dots \sqcup V_t\]
as guaranteed by the lemma, with $t_0 \leq t \leq T_0$. Now, with $d$ as chosen above, consider the cluster hypergraph $K(\varepsilon',d)$ of $H$ associated with this partition. Applying Lemma~\ref{codeg reg}, we obtain a subhypergraph $K'$ with 

\[\delta_2^+(K') \geq \left( \frac{1}{2} + \varepsilon - 48\sqrt[4]{2\varepsilon' + \frac{1}{t} + d}\right)\]
which contains at most $$\left( \frac{1}{\left(\frac{1}{2} + \varepsilon \right)^2} \cdot 288 \sqrt[4]{2\varepsilon' + \frac{1}{t} + d} \right) t < \left(1152 \sqrt[4]{2\varepsilon' + \frac{1}{t} + d}\right) t$$
isolated vertices.
Since we have chosen our constants so that $\varepsilon - 48\sqrt[4]{2\varepsilon' + \frac{1}{t} + d} > 0$ and $t_0$ is sufficiently large, we can apply Lemma~\ref{C tiling} to obtain a $\beta$-deficient $C_4^3$-tiling of the non-isolated vertices of $K'$. 

Now, note that a copy of $C_4^3$ in $K'$, say on vertices $1,2,3,4$ with $3$-edges $123, 234$, indicates four classes $V_1,V_2,V_3,V_4$ in $H$ such that $(V_1,V_2,V_3)$ and $(V_2,V_3,V_4)$ are $(\varepsilon, d)$-regular. We split $V_2, V_3$ into two equally sized classes, $V_2^1, V_2^2$ and $V_3^1,V_3^2$. It follows that $(V_1,V_2^1,V_3^1)$ and $(V_2^2,V_3^2,V_4)$ are $(\varepsilon,d)$-regular. Applying Lemma~\ref{reg path lemma} with $m = \frac{n}{2t}$, $(V_1,V_2^1,V_3^1)$
contains a loose path $P$ which spans all but at most $\frac{8 \varepsilon n}{2td} + 3 < \frac{6 \varepsilon n}{td}$ vertices of $(V_1,V_2^1,V_3^1)$, and analogously for $(V_2^2,V_3^2,V_4)$. Repeating this process across all copies of $C_4^3$ in our tiling of $K'$, we obtain a partial $\mathcal{P}$-tiling $\mathcal{T}$ of $H$ with the following properties:

\begin{enumerate}

\item $|\mathcal{T}| \leq \frac{T_0}{2}$;

\item $\mathcal{T}$ is  
$\left(1152 \sqrt[4]{2\varepsilon' + \frac{1}{t_0} + d} + \beta + \frac{3 \varepsilon'}{d} + \varepsilon' \right)$-deficient.

\end{enumerate}

The first property is clear, as $V(K') \leq T_0$ and thus $K'$ can be tiled with at most $\frac{T_0}{4}$ copies of $C_4^3$, each of which gives rise to two elements of $\mathcal{T}$. For the second property, we estimate the number of vertices omitted by a path-tiling produced as described. By Lemma~\ref{codeg reg}, $K'$ contains at most $\left(1152 \sqrt[4]{2\varepsilon' + \frac{1}{t_0} + d} \right)t$
isolated vertices, corresponding to at most $\left(1152 \sqrt[4]{2\varepsilon' + \frac{1}{t_0} + d} \right)n$ vertices of $H$. Additionally, $K'$ does not contain a vertex corresponding to $V_0$, so all vertices in $V_0$ are uncovered, contributing at most $\varepsilon' n$ omitted vertices. A $\beta$-deficient $C_4^3$-tiling of the non-isolated vertices of $K'$ yields at most $\beta t$ untiled vertices of $K'$ which are not isolated, corresponding to at most $(\beta t) \cdot \frac{n}{t} = \beta n$ vertices. Finally, each element of $\mathcal{T}$ covers all but at most $\frac{6 \varepsilon'n}{td}$ vertices from its corresponding clusters. Since there are at most $\frac{t}{2}$ such elements, they omit at most $\frac{3\varepsilon'n}{d}$ vertices in total from the clusters which correspond to tiled vertices in $K'$. Thus, the total number of vertices not covered by $T$ is at most 
\[\left( 1152 \sqrt[4]{2\varepsilon' + \frac{1}{t_0} + d} + \varepsilon' + \beta+ \frac{3\varepsilon'}{d} \right)n < \alpha n\]
by our choice of constants.
\end{proof}

\subsection{Proof of Theorem~\ref{looseHC}}\label{main thm subsection}

Finally, we are ready to prove Theorem \ref{looseHC}, using the lemmas developed in the previous subsections.

\begin{proof}[Proof of Theorem~\ref{looseHC}]

Fix $0 < \varepsilon < \frac{1}{2}$ .
We apply the Absorbing Lemma (Lemma~\ref{absorbing}) with $\varepsilon$ to obtain $n_1$ and $\beta = \frac{\varepsilon^6}{C}$ for which the lemma holds. We next apply the Reservoir Lemma (Lemma~\ref{reservoir}) with $\gamma = \frac{\beta}{2}$ to obtain $n_2$ such that the lemma holds. We shall set $\varepsilon_1 = \varepsilon - \varepsilon^3$, and $\varepsilon_2 = \varepsilon - \varepsilon^3 - \frac{\beta}{2}$. Note that $\varepsilon_1$ and $\varepsilon_2$ are both strictly positive (as $\beta < \varepsilon^6$). Finally, we apply the Path Tiling Lemma (Lemma~\ref{path tiling lemma}) with $\varepsilon_2$ and $\alpha = \frac{\beta}{2}$ to obtain $n_3, p$ such that the lemma holds. With these constants determined, we proceed as follows.

Let $n_0 \geq \max\{n_1, 2n_2, 2n_3\}$, with $n_0$ also chosen to be large enough that \[\frac{1}{12}\left(\frac{\beta}{2}\right)^3 \left(1 - \varepsilon^3\right)n_0 \geq p + 1.\]
Let $H$ be an $n$-vertex $3$-graph with no isolated vertices, $n \geq n_0$, $n \in 2\mathbb{N}$, and $\delta_2^+(H) \geq \left( 
\frac{1}{2} + \varepsilon \right)n$. Applying the Absorbing Lemma, we obtain a loose path $P$ with $|V(P)| \leq \varepsilon^3 n$ which can absorb any set $U \subset V(H) \setminus V(P)$ with $|U| \in 2\mathbb{N}$ and $|U| \leq \beta n$. Let $H_1$ be the $3$-graph induced on $V(H) \setminus V(P)$. Note that $V(H_1) \geq (1 - \varepsilon^3)n_0 > n_2$ and $\delta_2^+(H_1) \geq \left( \frac{1}{2} + \varepsilon_1 \right)|V(H_1)|$; also note that $\varepsilon_1 > \frac{\beta}{2}$. Thus, we can apply the Reservoir Lemma to obtain a reservoir $R$ of size at most $\frac{\beta}{2}|V(H_1)|$ which is capable of connecting at least 
\[\frac{1}{12}\left(\frac{\beta}{2}\right)^3 V(H_1) \geq \frac{1}{12}\left(\frac{\beta}{2}\right)^3 \left(1 - \varepsilon^3\right)n_0\]
mutually disjoint pairs of vertices. Finally, let $H_2$ be the $3$-graph induced on $V(H_1)\setminus V(R)$. Observe that $V(H_2) \geq \left(1 - \varepsilon^3 - \frac{\beta}{2}\right)n_0 > n_3$ and $\delta_2^+(H_2) \geq \left(\frac{1}{2} + \varepsilon_2 \right)|V(H_2)|$. Thus, we can apply the Path Tiling Lemma to $H_2$ to obtain a $\frac{\beta}{2}$-deficient path-tiling of $H_2$ by at most $p$ loose paths. Now, together with our absorbing path, we have $p+1$ mutually disjoint loose paths, whose endpoints we can connect using $R$ to form a loose cycle. This loose cycle spans all but at most 
\[|R| + \frac{\beta}{2}|V(H_2)| < \frac{\beta}{2}n + \frac{\beta}{2}n = \beta n\] 
vertices, and since any loose cycle contains an even number of vertices, the set of vertices which are not spanned by this loose cycle must also have even cardinality. Thus, $P$ can absorb all uncovered vertices, yielding a loose Hamiltonian cycle in $H$.
\end{proof}

\section{Conclusion}\label{further questions}

As with other hypergraph versions of minimum degree, we are able to obtain Dirac-type results and perfect matching thresholds for minimum positive co-degree. The similarity of Construction~\ref{main construction} to certain constructions which are extremal in the graph case is notable.
The minimum positive co-degree parameter was recently introduced as a ``reasonable notion'' of minimum degree for hypergraphs, and our investigation was partially motivated by the desire to better understand $\delta_{r-1}^+(H)$ as an analog to minimum degree in graphs. In the cases which we investigated, it seems that minimum positive co-degree indeed behaves similarly to minimum vertex degree in $2$-graphs: not only do our positive co-degree thresholds naturally align with minimum degree thresholds for analogous structures in graphs, but the possible obstructions to finding spanning structures in the two settings are similar. 

Many avenues remain open. In the vein of perfect matchings, one could also consider positive co-degree thresholds for various types of $\mathcal{K}$-tiling.  Lemma~\ref{C tiling} demonstrates the potential utility of tiling results in this setting; there are also some types of tiling which would be interesting to study in their own right. Finding thresholds for clique factors seems a natural starting point. 

Theorem~\ref{looseHC} can potentially be improved: Construction~\ref{main construction} achieves a minimum positive co-degree of $\frac{n}{2} - 1$, and we conjecture that $\frac{n}{2}$ is indeed the positive co-degree threshold for loose Hamiltonian cycles in $3$-graphs. A sharpening of the asymptotic threshold to an exact one would be interesting, as would a generalization of Theorem~\ref{looseHC} to $r$-graphs with $r \geq 4$. Other types of Hamiltonian cycle could also be considered. Another related question which may have bearing upon sharpening Theorem~\ref{looseHC} is that of stability: how many $3$-graphs $H$ which avoid a loose Hamiltonian cycle have $\delta_2^+(H)$ close to $\frac{n}{2} - 1$? We know that Construction~\ref{main construction} is not the unique such $3$-graph. Some subgraphs of Construction~\ref{main construction} also achieve minimum positive co-degree $\frac{n}{2} - 1$, and the disjoint union of two copies of $K_{n/2}$ has minimum positive co-degree $\frac{n}{2} - 2$. It is unclear whether these should be the only obstructions to a loose Hamiltonian cycle near minimum positive co-degree $\frac{n}{2}$.

\section{Acknowledgements}

The authors would like to thank the organizers and lecturers of the 2023 Summer School in Graph Coloring at the University of Illinois, Urbana-Champaign for providing an excellent introduction to the absorbing method, which inspired and informed this project. The authors would also like to thank Cory Palmer and Ryan Wood for helpful discussions and feedback on the manuscript.

\printbibliography

\end{document}